\documentclass[twoside,leqno]{article}

\usepackage[letterpaper]{geometry}

\usepackage{siamproceedings}

\usepackage[T1]{fontenc}
\usepackage{amsfonts}
\usepackage{graphicx}
\usepackage{epstopdf}
\usepackage{enumitem}
\usepackage{algorithmic}
\ifpdf
  \DeclareGraphicsExtensions{.eps,.pdf,.png,.jpg}
\else
  \DeclareGraphicsExtensions{.eps}
\fi

\newsiamremark{remark}{Remark}
\newsiamremark{hypothesis}{Hypothesis}
\crefname{hypothesis}{Hypothesis}{Hypotheses}
\newsiamthm{claim}{Claim}
\newsiamremark{example}{Example}
\newsiamremark{question}{Question}

\usepackage{amsopn}

\begin{document}

\newcommand\relatedversion{}
%\renewcommand\relatedversion{\thanks{The full version of the paper can be accessed at \protect\url{https://arxiv.org/abs/0000.00000}}} % Replace URL with link to full paper or comment out this line

%\title{\Large SIAM/ACM Proceedings Macros for Use With LaTeX\relatedversion}
  %  \author{James E. Haines\thanks{Society for Industrial and Applied Mathematics (\email{jhaines@siam.org}, \email{rginder@siam.org}, \url{http://www.siam.org}).}
   % \and Rachel Ginder\footnotemark[2]    
   % \and Mitch C. Donovan\thanks{Department of Applied Mathematics, Fictional University, Austin, TX and NASA, Johnson Space Center, Houston, TX
  %(\email{donovan@fictional.edu}, \email{mcd@nasa.gov}, \url{https://www.nasa.gov/johnson/}).}}

\title{\Large Symplectic Weyl Laws\relatedversion}
    \author{Daniel A. Cristofaro-Gardiner\thanks{University of Maryland, College Park (\email{dcristof@umd.edu}).}}

\date{}

\maketitle

% Copyright Statement
% When submitting your final paper to a SIAM proceedings, it is requested that you include
% the appropriate copyright in the footer of the paper.  The copyright added should be
% consistent with the copyright selected on the copyright form submitted with the paper.
% Please note that "20XX" should be changed to the year of the meeting.

% Default Copyright Statement
\fancyfoot[R]{\scriptsize{Copyright \textcopyright\ 20XX by SIAM\\
Unauthorized reproduction of this article is prohibited}}

% Depending on which copyright you agree to when you sign the copyright form, the copyright
% can be changed to one of the following after commenting out the default copyright statement
% above.

%\fancyfoot[R]{\scriptsize{Copyright \textcopyright\ 20XX\\
%Copyright for this paper is retained by authors}}

%\fancyfoot[R]{\scriptsize{Copyright \textcopyright\ 20XX\\
%Copyright retained by principal author's organization}}

%\pagenumbering{arabic}
%\setcounter{page}{1}%Leave this line commented out.

%\begin{abstract} The abstract must be able to stand alone and so cannot contain citations to the paper's references, equations, etc.  An abstract must consist of a single paragraph and be concise. Because of online formatting, abstracts must appear as plain as possible. Any equations should be inline.
%\end{abstract}

\begin{abstract}  We survey a number of Weyl type laws that have recently been established in low-dimensional symplectic geometry.  These have had a number of applications, which we also introduce.  We sketch a number of proofs so that the reader can get a sense of how these formulas are proved and how they can be applied. 
%explain.
\end{abstract}

\section{Introduction.}
%The introduction introduces the context and summarizes the manuscript. It is important to clearly state the contributions of this piece of work.

In 1911, Weyl proved the following influential result: 
% regarding the asymptotics of the eigenvalues of the Laplacian.  
let $\Omega \subset \mathbb{R}^d$ be a bounded domain, let $N(T)$ be the number of Laplace eigenvalues, with Dirichlet boundary conditions, not more than $T$ and counted with multiplicity, and let $\omega_d$ denote the volume of the unit ball.  Then
\begin{equation}
\label{eqn:weylc}
\lim_{T \to \infty} \frac{N(T)}{T^{d/2}} = \frac{\omega_d}{ (2 \pi)^d } \text{vol}(\Omega).
\end{equation}
The famous formula \eqref{eqn:weylc} is now known as ``Weyl's law" and has inspired generations of research. 

In recent years, some formulas reminiscent of \eqref{eqn:weylc} have been found in the context of low-dimensional symplectic and contact geometry and these have had important applications.  The aim of this article is to give an accessible account for how this works; an emphasis is providing outlines of key arguments when appropriate so that the reader can get a sense of how these laws are proved and how they can be used.  

The applications we discuss are to problems in dynamics, problems about the algebraic structure of certain groups of homeomorphisms of surfaces, questions about symplectic packing problems, and questions about ``hidden boundaries" of symplectic manifolds; it is our hope that anyone interested in any of these areas will find these notes helpful.  Many of the problems we discuss are stories in their own right and Weyl laws are just one component of the proof, so for brevity we mainly focus on those aspects of the proofs for which the Weyl laws are essential, trying to give a sense for other parts of the arguments when we can.

\begin{remark}
\label{rmk:minimal}
Though it is completely beyond the scope of these notes, we would be remiss if we did not mention that there is also a remarkable story about Weyl laws in the context of minimal surface theory, see e.g. \cite{lmn}, with many intriguing parallels. 
\end{remark}

\subsection{Outline of the article}

%The Weyl laws that we will discuss are in dimensions $2, 3$ and $4$ and the main developments we survey have taken place over a period of about $15$ years.  
The first known symplectic Weyl laws were in dimensions $3$ and $4$ and a certain three-manifold invariant called embedded contact homology (ECH) is central to that story.  Thus we start by reviewing those developments and reviewing ECH; this is the topic of Section~\ref{sec:ch1} and our discussion here includes both the Well laws themselves and some relevant applications.  Next, there was a flurry of interest in the case of surfaces and this is the topic of Section~\ref{sec:ch2}: here the relevant Weyl laws involve a cousin of ECH called periodic Floer homology, as well as a kind of variant of Heegaard Floer cohomology leading to the construction of ``link spectral invariants".  
%Connections with continuous symplectic geometry were central to many of these developments and though for brevity we can not say a huge amount of this, we try to give some sense of how this works.  
Next in Section~\ref{sec:ch3} we survey some of the most recent related developments; these range from applications of the subleading asymptotics to significant progress on the topological invariance in helicity.  We end by saying a few words about the subleading asymptotics and the situation in higher dimensions.

%\subsection{Acknowledgments.}  Conversations, readers etc.

\section{The first symplectic Weyl laws and some of their applications.}
\label{sec:ch1}

The origins of the first symplectic Weyl laws are in the context of symplectic embedding problems.  Thus we first give a sense for the relevant questions in this area.  We then state the relevant results and explain some key applications.

\subsection{Symplectic capacity theory.}

Recall that a {\em symplectic manifold} is a pair $(X,\omega)$, where $X$ is a $2n$-dimensional manifold and $\omega$ is 
 %Let $X$ be a $2n$-dimensional manifold.  Recall that a {\em symplectic form} $\omega$ on $X$ is
 a closed two-form such that the $n$-fold wedge-product $\omega^n$ is nowhere vanishing, called a {\em symplectic form}.   Symplectic manifolds are the natural setting for Hamilton's equations of motion  and therefore of fundamental importance.  We will see numerous applications to Hamiltonian dynamics in these notes.
 
 The following seemingly naive question turns out to be of central importance: {\em how big is $(X,\omega)$?}  One natural size measurement is the {\em volume} $\text{vol}(X,\omega) := \frac{1}{n!} \int_X \omega^n.$
 This is certainly an important quantity to consider.
 %a symplectic embedding $(X_1,\omega_1) {\stackrel{s}\hookrightarrow} (X_2,\omega_2)$ implies that $\text{vol}(X_1,\omega_1) \le \text{vol}(X_2,\omega_2)$.
 % For example, the natural definition of a {\em symplectic embedding} from one symplectic manifold $(X_1,\omega_1)$ into another $(X_2,\omega_2)$ is a smooth embedding $\psi: X_1 \hookrightarrow X_2$, such that $\psi^* \omega_2 = \omega_1$; let us denote a symplectic embedding by ${\stackrel{s}\hookrightarrow}$.
 %Then, a symplectic embedding $(X_1,\omega_1) {\stackrel{s}\hookrightarrow} (X_2,\omega_2)$ implies that $\text{vol}(X_1,\omega_1) \le \text{vol}(X_2,\omega_2)$ and in particular, the volume is an embedding obstruction.
 However, volume alone does not tell the whole story.  For example, 
 %Gromov has shown the following striking result.  
 consider the {\em ball} and the {\em infinite cylinder}, defined by 
 %symplectic manifolds
 \[ B_n(r) = \left \lbrace \pi \left(\frac{|z_1|^2}{r} + \ldots + \frac{|z_n|^2}{r}\right) < 1 \right \rbrace \subset \mathbb{C}^n, \quad \quad Z_n(R) =  \left \lbrace   \pi \frac{|z_1|^2}{R} < 1 \right \rbrace \subset \mathbb{C}^n.\]
These are symplectic manifolds: the symplectic form is given by identifying $\mathbb{C}^n = \mathbb{R}^{2n}$ via $z_i = x_i + y_i$ and restricting the {\em standard symplectic form} $\omega_{std} = dx_1 \wedge dy_1 + \ldots + dx_n \wedge dy_n$ on $\mathbb{R}^{2n}$.  
%Thus, $B_n$ is a (symplectic) ball and $Z_n$ is an    
Then the volume of $B_n(r)$ is finite while the volume of $Z_n(R)$ is infinite,  however the famous Gromov nonsqueezing theorem states the following:

\begin{theorem} \cite{gromov}
\label{thm:gromov}
There is a symplectic embedding $B_n(r)  {\stackrel{s}\hookrightarrow} Z_n(R)$ if and only if $r \le R.$
\end{theorem}

In other words, the volume measurement alone does not capture the ``symplectic" size of the ball and the cylinder.  One would therefore like to know the following: {\em What other kinds of symplectic size measurements exist?}   To put this in a more axiomatic framework, let us define a {\em symplectic capacity in dimension $2d$} to be a rule $c$ that assigns to each symplectic $2d$ dimensional manifold $(X,\omega)$ a number $c(X,\omega) \in \mathbb{R}_{\ge 0} \cup \lbrace \infty \rbrace$, subject to the {\em Monotonicity Axiom}, which says that the existence of a symplectic embedding $(M_1,\omega_1)  {\stackrel{s}\hookrightarrow} (M_2,\omega_2)$ of $2d$-dimensional symplectic manifolds implies that 
\begin{equation}
\label{eqn:mono}
c(M_1,\omega_1) \le c(M_2,\omega_2),
\end{equation}
and the {\em Scaling Axiom}, which says that $c(X, \alpha \cdot \omega ) = \alpha c (X,\omega)$ for any $\alpha > 0$.  To distinguish from quantities derived from the volume measurement, we will call such a symplectic capacity {\em Non-Trivial} if $c(Z_{d}) < \infty$.  The following is the first example of a symplectic capacity:

\begin{example}
\label{ex:gromov}
Let $(X,\omega)$ be a $2d$-dimensional symplectic manifold.   Define the {\em Gromov Width}
\[ c_{Gr}(X,\omega) = sup \lbrace r | B_d(r)  {\stackrel{s}\hookrightarrow} (X,\omega) \rbrace.\]
It is immediate from the definition that $c_{Gr}$ satisfies the Monotonicity and Scaling axioms.  The Non-Triviality axiom follows from Theorem~\ref{thm:gromov}.  Thus, the Gromov Width is a non-trivial symplectic capacity.
\end{example}

\subsection{ECH capacities.}
\label{sec:packing}

We saw our first example of a symplectic capacity in Example~\ref{ex:gromov}.  In fact there are many more.  For example, the Gromov width can be generalized from the ball to any open domain; these are called {\em embedding capacities}.
%to embeddings of any fixed manifold.
 %is an example of an; we can generalize this from the ball to any open domain $(Y,\omega_Y)$ in $\mathbb{R}^{2n}$ by defining $c_Y(X,\omega) = sup \lbrace r | (Y, r \cdot \omega_y) {\stackrel{s}\hookrightarrow} (X,\omega) \rbrace$. 
 For more examples of capacities, we refer the reader to the beautiful article \cite{qsg}.

There is a particularly interesting family of symplectic capacities of four-manifolds, called {\em ECH capacities}.  The ECH capacities are a sequence
\[ 0 = c_0(X,\omega) \le c_1(X,\omega) \le c_2(X,\omega) \le \ldots \le \infty, \]
associated to any symplectic $4$-manifold, defined by Hutchings \cite{qech}.   We will defer the definition until \S\ref{sec:defnech} and first focus on some key properties.   

First of all, the ECH capacities are {\em strong} invariants, in that they give informative symplectic size measurements.  For example, as an important test case, let us consider {\em symplectic ball packing problems}, i.e.. the question of whether $(M_1,\omega_1)  {\stackrel{s}\hookrightarrow} (M_2,\omega_2)$ when $M_1$ is a disjoint union of symplectic balls.  We can in particular study the {\em symplectic ball-packing numbers} $p_k(M,\omega)$ of a $2n$-dimensional symplectic manifold; this is the maximum ratio of the volume of $(M,\omega)$ that can be filled by $k$ equally sized symplectic balls.  These can be very subtle.  For example, the following are the seemingly mysterious packing numbers of the four-ball.
% see e.g. \cite{biran}. 
 First of all $p_k(B_4) = 1$ for $k \ge 9$.  For $k \le 8$, we have:
\[
\begin{array}{c|cccccccc}
k & 1 & 2 & 3 & 4 & 5 & 6 & 7 & \geq 8 \\ \hline
p_k(B_4) & 
1 & \tfrac{1}{2} & \tfrac{3}{4} & 1 & 
\tfrac{20}{25} & \tfrac{24}{25} & \frac{63}{64} & \tfrac{288}{289}
\end{array}
\]
In fact, the ECH capacities are strong enough to see this: \cite{qech} shows that if $(M_1,\omega_1) = B_4(a_1) \sqcup \ldots \sqcup B_4(a_n)$ and $(M_2,\omega_2) = B_4(b)$, 
there is a symplectic embedding $(M_1,\omega_1)  {\stackrel{s}\hookrightarrow} (M_2,\omega_2)$
%\]
if and only if 
%[ 
$c_k(M_1,\omega_1) \le c_k(M_2,\omega_2) \quad$ for all $k$.
% \forall k.$
%\]

%\begin{theorem}\cite{qech}
%Let $(M_1,\omega_1) = B_4(a_1) \sqcup \ldots \sqcup B_4(a_n)$ and $(M_2,\omega_2) = B_4(b)$.  There is a symplectic embedding $(M_1,\omega_1)  {\stackrel{s}\hookrightarrow} (M_2,\omega_2)$
%\]
%if and only if 
%[ 
%$c_k(M_1,\omega_1) \le c_k(M_2,\omega_2) \quad \forall k.$
%\]
%\end{theorem}

The second important property of ECH capacities are that they are {\em computable}, in examples of interest.  For example, define the {\em symplectic ellipsoid} $E(a,b) = \left \lbrace \pi \frac{ |z_1|^2}{a} + \pi \frac{ |z_2|^2}{b} < 1 \right \rbrace \subset \mathbb{C}^2,$
where the symplectic form is given by the restriction of the standard form, as above.  These generalize the examples from Gromov non-squeezing, since $B_4(a) = E(a,a)$ and $Z_4(R) = E(R,\infty)$. 

\begin{proposition} \cite{qech}
% \cite{hutchingsellip}
\label{prop:echell}
$c_k(E(a,b))$ is the $(k+1)^{st}$ smallest entry in the matrix $(ma + nb)_{ (m,n) \in \mathbb{Z}_{\ge 0} \times \mathbb{Z}_{\ge 0} }$.  
\end{proposition}

\begin{example}
\label{ex:ball}
The ECH capacities $c_k$ of the ball are given by
\[ 0, 1, 1, 2, 2, 2, 3, 3, 3, 3, \ldots .\]
\end{example}

\subsection{The ECH capacity Weyl law.}

As we explained in the previous section, the ECH capacities are strong, computable embedding obstructions.  The most basic embedding obstruction is the volume obstruction, so one could ask: {\em are the ECH capacities strong enough to detect volume?}

The following theorem gives the state of the art of this.  An important class of symplectic 4-manifolds are ``Liouville domains". To define these, let us recall that a {\em contact form} on a $(2n+1)$-dimensional manifold is a one-form $\lambda$ such that $\lambda \wedge d \lambda^n$ is nowhere vanishing.  Contact forms are odd-dimensional cousins of symplectic forms.   A {\em Liouville domain} is a compact exact symplectic manifold $(X,\omega)$ such that $\omega = d \lambda$, where $\lambda|_{\partial X}$ is a contact form compatible with the positive boundary orientation.  For example, balls, cylinders and ellipsoids are all Liouville domains.
%for example.. 

We now state our first symplectic Weyl law:

\begin{theorem}\cite{cghr}
\label{thm:cghr}
Let $(X,\omega)$ be a four-dimensional Liouville domain with all ECH capacities finite.  Then
\begin{equation}
\label{eqn:weyl}
\lim_{k \to \infty} \frac{c^2_k(X,\omega)}{k} = 2 \int_X \omega \wedge \omega.
\end{equation}
\end{theorem}

We will try to give a sense of the proof later in the notes; it is involved and requires explaining quite a bit more about ECH.  For now, however, let us give a sketch of the proof in an important and more accessible special case reminiscent of the setup considered by Weyl.  To give the proof, we will need another axiom of ECH capacities, namely the Disjoint Union Axiom, which says that for a disjoint union $X \sqcup Y$, we have
\begin{equation}
\label{eqn:disjoint}
c_k(X \sqcup Y) = \max_{i + j = k} c_i(X) + c_j(Y).
\end{equation}

\begin{theorem}\cite{qech}
%[Hutchings]
The Weyl law \eqref{eqn:weyl} holds in the special case where $X \subset \mathbb{R}^4$ is a bounded domain with smooth boundary.
\end{theorem}

\begin{proof}[Sketch of proof]

{\em Step 1.}  One first verifies by direct computation using the result of Example~\ref{ex:ball} that the Weyl law holds for any ball.  Thus, by the Disjoint Union Axiom \eqref{eqn:weyl} and an elementary argument, the Weyl Law holds for any disjoint union of balls.

{\em Step 2.}  Now, given $X$, we can fill $X$ by a disjoint union of balls, up to an error of $\epsilon > 0$ in the volume.  Thus by applying the Weyl Law for disjoint unions of balls established in Step 1, and the Monotonicity Axiom \eqref{eqn:mono}, we get that   
\[ \lim_{k \to \infty} \frac{c^2_k(X,\omega)}{k} \ge 2 \int_X \omega \wedge \omega,\]
since $\epsilon$ was arbitrary.

{\em Step 3.}  We can similarly include $X$ into a large ball, and fill the complement of $X$ in this ball by balls, up to an error of $\epsilon > 0$ in the volume.  Thus by again applying the Weyl Law from Step 1 and the Monotonicity Axiom, as in Step 2, we get the upper bound that matches the lower bound from Step 2, hence \eqref{eqn:weyl}.

\end{proof}

\subsection{Embedded Contact Homology}
\label{sec:defnech}

To say more about how the ECH capacities $c_k$ are defined and to give a sense of why they satisfy their key properties, we now introduce an invariant of contact three-manifolds, called {\em embedded contact homology}, defined by Hutchings.    As we aim only to give a brief overview, we do not always give full details; we refer the reader to \cite{echlecture} for more detail about any of the terms below.

Let $(Y,\lambda)$ be a closed three-manifold with contact form. 
% (Recall that this means that $\lambda \wedge d \lambda > 0$.)  
The contact form $\lambda$ defines a canonical vector field, the {\em Reeb vector field} $R$, by the equations
\[ d \lambda(R,\cdot) = 0, \quad \quad \lambda(R) = 1.\]
The (unparametrized) periodic orbits of $R$ are called {\em Reeb orbits}.  The contact form also defines a {\em contact structure} $\xi := \text{Ker}(\lambda).$  The embedded contact homology $ECH(Y,\lambda)$ (with $\mathbb{Z}/2\mathbb{Z}$-coefficients) is the homology of a chain complex $ECC(Y,\lambda,J)$.  The chain complex is freely generated by finite sets $\lbrace (\alpha_i,m_i ) \rbrace$, called {\em ECH generators}, where the $\alpha_i$ are distinct embedded Reeb orbits and the $m_i$ are positive integers, such that $m_i = 1$ if $\alpha_i$ is a hyperbolic orbit; 
here, we are assuming the contact form $\lambda$ is nondegenerate, which holds generically, and we refer to \cite{echlecture} for the relevant definitions.  

To define the chain complex differential $\partial$ we need to review Gromov's theory of pseudoholomorphic curves \cite{gromov}.  Let $(X,\omega)$ be any symplectic manifold.  Gromov observed that the existence of the symplectic structure implies that there always exists an {\em almost complex structure}, i.e. a bundle map $J : TX \to TX$ such that $J^2 = -1$.  Moreover, the space of all such $J$ is non-empty and contractible if we assume in addition that $J$ is {\em compatible with a symplectic form}, meaning that the rule $g(u,v) = \omega(u,Jv)$ defines a Riemannian metric.  Gromov then showed that it is profitable to study {\em $J$-holomorphic maps}, namely maps $u : (\Sigma,j) \to (X,J)$ satisfying the equation
\begin{equation}
\label{eqn:jhol}
du \circ j = J \circ du;
\end{equation}
here, $(\Sigma,j)$ is a Riemann surface.  
%When the map $u$ is {\em somewhere injective}, namely there is a point $z$ with $u^{-1} u(z) = \lbrace z \rbrace$ and $du_z$ injective, the space of $J$-holomorphic curves is particularly nice; for example, if $X$ is closed and $J$ is generic, then this space is a manifold near such a somewhere injective curve $u$.  
Returning to the problem at hand, we can take $X = \mathbb{R}_s \times Y$ and then the contact structure defines a symplectic form by the rule $d(e^s \lambda)$; we call this symplectic manifold the {\em symplectization} of $(Y,\lambda)$.  The almost complex structures we will want to study are 
%{\em admissible}
{\em $\lambda$-compatible} 
 in the sense that they send $\partial_s$ to $R$, fix $\xi$, compatibly with $d \lambda$, and are $\mathbb{R}$-invariant; such a $J$ is automatically compatible with $d(e^s \lambda)$.  We will want our Riemann surfaces to be finitely punctured
 %  i.e. they are given by compact Riemann surfaces, possibly with multiple components, minus a finite number of points, 
 and we demand that the map $u$ is asymptotic to Reeb orbits at the punctures.
 %; see \cite{echlecture} for the precise definition.  
We now define the chain complex differential $\partial$ by the rule
\[ \langle  \partial \alpha, \beta \rangle := \# \mathcal{M}^{I=1}_J(\alpha,\beta).\]
To elaborate, here $\mathcal{M}$ is the space of $J$-holomorphic currents, modulo translation in the $\mathbb{R}$ direction, and $J$ is chosen generically; a $J$-holomorphic current is a map $u$ satisfying \eqref{eqn:jhol}, modulo equivalence of currents; and, the condition $\alpha$ means that $\alpha$ gives the asymptotics of $u$ as $s \to \infty$, and analogously for $\beta$ as $s \to - \infty$.
%for the precise definition, see \cite{echlecture}
% so that all somewhere injective curves are transverse.

Finally, to explain the condition $I  = 1$, we need to introduce the {\em ECH index} $I$ given by the formula
\begin{equation}
\label{eqn:defnech}
I(C) =  c_\tau(C) + Q_\tau(C) + CZ^I_\tau(\alpha)  - CZ^I_{\tau}(\beta).
\end{equation}
We refer to \cite{echlecture} for the precise definitions, noting that the key property of the ECH index for our purposes is that low index ECH curves tend to be embedded, which is the genesis of the name of the theory.  In particular, one can show that an $I = 1$ current has a single embedded component that is rigid modulo translation, and all other (possibly empty) components are covers of $\mathbb{R}$-invariant cylinders over orbit sets.

A proof that $\partial$ is well-defined and satisfies $\partial^2 = 0$ can be found in \cite{d2=0}.   Hence, the homology $ECH(Y,\lambda,J)$ is defined.  Moreover, Taubes has shown \cite{taubes} that there is a canonical isomorphism
\begin{equation}
\label{eqn:echswf}
ECH(Y,\lambda,J) \simeq \widehat{HM}(Y)
\end{equation}
between embedded contact homology and the {\em Seiberg-Witten Floer cohomology} from \cite{km}.  This implies, in particular, that ECH is independent of the choice of almost complex structure $J$, and even independent of the contact form.  
%We will return to the isomorphism \eqref{eqn:echswf} when we explain more about the proof of Theorem~\ref{thm:cghr}.  
One consequence of the fact that the isomorphism is canonical is we can speak of a well-defined ECH class $\sigma$, independent of any choices.  In the next section we will associate a ``spectral invariant" to such a class that will play a key role.
% in our story.
   
\subsection{ECH spectral invariants, the $U$-map, and the definition of ECH capacities}
\label{sec:spectral}

To define the ECH capacities, we first explain a construction for associating a nonnegative real number to a nonzero class in ECH.  The construction makes use of the fact that ECH has a canonical filtration induced by the {\em action} of an ECH generator $\alpha$
\[ \mathcal{A}(\alpha) = \sum_i m_i \int_{\alpha_i} \lambda.\]
The construction also serves as a model for extracting numerical invariants from filtered Floer theory, a theme we will return to in different contexts throughout these notes

To elaborate, it is not hard to show, using the equation \eqref{eqn:jhol}, that the differential $\partial$ decreases $\mathcal{A}$, so we can define $ECH^L$ to be the homology of the subcomplex generated by ECH generators with action $\le L$, and there is then an inclusion induced map $ECH^L \to ECH.$
Now if $\lambda$ is nondegenerate and  $0 \ne \sigma \in ECH(Y,\lambda)$, we can define the associated {\em spectral invariant} to be the first filtration level where $\sigma$ appears, i.e. the smallest $L$ such that $\sigma$ is in the image of above map.  Concretely, this means that 
% be such a class.  We define the associated {\em spectral invariant} 
$c_{\sigma}(\lambda)$ is the minimum amount of action required to represent the class $\sigma$: we look at all possible cycles representing $\sigma$ and take the action of the cycle requiring the least amount of action to represent $\sigma$.  When $\lambda$ is degenerate, one can still define $c_{\sigma}(\lambda)$ by approximating with nondegenerate forms, even in the $C^0$ topology, and taking a limit.  The reason that this is well-defined is due to the existence of ECH cobordism maps and we will say a few more words about this when we discuss some of the properties of the ECH capacities below.

To define ECH capacities, the idea is then to pick out a distinguished sequence of classes $\sigma$ and look at the associated spectral invariants.  To do this, we need to introduce some additional structures on $ECH$.  First of all, it follows by an easy exercise that the empty orbit set $\emptyset$ is a cycle, so there is a canonical (possibly vanishing) class in ECH given by $[\emptyset]$.  Next, there is a map 
\[ U : ECH(Y,\lambda) \to ECH(Y,\lambda),\]
which is defined analogously to $\partial$, except that it counts $I = 2$ currents through a marked point $(0,z)$ chosen away from all Reeb orbits; it agrees with an analogous $U$-map on $HM$ via the isomorphism \eqref{eqn:echswf}.  One can now define the {\em ECH spectrum} $c_k(Y,\lambda) = min \lbrace c_\sigma(\lambda) | U^k \sigma = [ \emptyset ], \sigma \ne 0 \rbrace.$
Finally, given a Liouville domain $(X,\omega)$ we define 
\[ c_k(X,\omega) = c_k(Y,\lambda) \]
where $Y = \partial X$ and $\lambda$ is any primitive of $\omega$ that restricts to $Y$ as a contact form; it is not hard to see that this is well-defined \cite{qech}.  
%We note that the condition that $X$ is a Liouville domain forces $[\emptyset] \ne 0$, though for brevity we will not explain the proof of this here; it is an easy exercise given some properties of ECH cobordism maps that will not be our focus in this exposition.  

Let us give a sense for why the Monotonicity and Scaling Axioms for symplectic capacities hold:

\begin{proof}[Sketch of Scaling and Monotonicity]

The Scaling Axiom is an easy consequence of the fact that there is a canonical chain complex isomorphism $ECC(Y,r \cdot \lambda,J) = ECC(Y, \lambda,J)$.  

As for Monotonicity, this is much deeper.  It is a consequence of the isomorphism \eqref{eqn:echswf} that an oriented cobordism from $Y_1$ to $Y_2$ induces a map on ECH, since this is known to hold for $\widehat{HM}$ by work of Kronheimer-Mrowka \cite{km}. Hutchings-Taubes prove that if this is a  {\em weakly exact symplectic} cobordism $(X,\omega)$ from one contact manifold $(Y_1,\lambda_1)$ to another $(Y_2,\lambda_2)$, i.e $\omega$ is exact and its restriction to the boundary satisfies $\omega = d \lambda_i$,
% for a one form $\lambda$ restricting to $\lambda_1$ and $\lambda_2$, 
then the associated cobordism map counts $J$-holomorphic buildings, in the sense that a nonzero cobordism map coefficient from an ECH generator $\alpha$ to another generator $\beta$ implies the existence of a $J$-holomorphic building from $\alpha$ to $\beta$; a $J$-holomorophic building is a slight generalization of a $J$-holomorphic current, where one is allowed to have multiple levels, analogously to a ``broken flow line" in Morse homology.  It then follows from this $J$-holomorophic curve property, by the same argument as above used to show that the chain complex differential decreases action, that the cobordism map decreases action.  The Monotonicity Axiom now follows from the fact that given a symplectic embedding $(X_1,\omega_1)   {\stackrel{s}\hookrightarrow} (\text{int}(X_2),\omega_2)$, with $\text{int}$ denoting the interior, one can remove the image of the interior of $X_1$ to get a symplectic cobordism from $\partial X_2$ to $\partial X_1$.
\end{proof}

We remark that the continuity of the ECH spectral invariants under $C^0$ deformation of the contact form mentioned above, used to define spectral invariants in the degenerate case, is proved by a very similar argument.

%ECH$^ by a similar argument 

%\subsection{Example: the ECH capacities of an ellipsoid.}

Let us now briefly outline the proof of Proposition~\ref{prop:echell}; note that the proposition implies the Non-Triviality Axiom, since the infinite cylinder $Z$ is an $E(1,\infty)$.  

\begin{proof}[Sketch of proof of Proposition~\ref{prop:echell}]
Let $b/a$ be irrational.  One first computes that $E(a,b)$ has exactly two embedded Reeb orbits, namely $\gamma_1 = \lbrace z_2 = 0 \rbrace$ and $\gamma_2 = \lbrace z_1 = 0 \rbrace$, both are elliptic, and the contact form is nondegenerate; $\gamma_1$ has action $a$ and $\gamma_2$ has action $b$.   Thus the ECH generators are all of the form $\alpha = \lbrace (\gamma_1, m), (\gamma_2, n) \rbrace$.  One next computes that the ECH index of such a generator is given by
\begin{equation}
\label{eqn:Iellip}
 I(\alpha,\emptyset) = m + n + 2mn + \sum^m_i(2 \lfloor i a/b \rfloor + 1) +  \sum^n_j(2 \lfloor j b/a \rfloor + 1).
 \end{equation}
for the details, see \cite{echlecture}; here we are using the fact that there is a unique relative homology class between any two orbit sets in $S^3$ in order to refer to the ECH index relative to two generators without specifying a homology class.  It follows from this equation that $I$ is always even, so the chain complex differential vanishes, any such $\alpha$ gives a nonzero class, and the associated spectral invariant is the action of this class.  It is now very helpful to interpret the right hand side of \eqref{eqn:Iellip} as the number of integer lattice points in the triangle bounded by the axes and the line through $(m,n)$ with slope $-b/a$.  This implies in particular that for any even nonnegative integer $2k$, there is exactly one generator with $I = 2k$, and these generators come in order of their action.  This proves Proposition~\ref{prop:echell} in the case where $b/a$ is irrational and the general case then follows immediately from continuity.
%Next one computes that
\end{proof}

\subsection{A Weyl law for contact three-manifolds that implies Theorem~\ref{thm:cghr}}
\label{sec:contactweyl}

To prove Theorem~\ref{thm:cghr}, we prove a more general result for contact $3$-manifolds that implies it; this result is also quite useful in applications.

Let us now state this result.  To state it, we first have to introduce the grading on ECH.  It is induced by the ECH index $I$ from \eqref{eqn:defnech}.  One attempts to define
%\begin{equation}
%\label{eq:relative}
$I(\alpha,\beta) = I(Z) $
%\end{equation}
where $Z$ is any relative homology class from $\alpha$ to $\beta$; one computes that when $c_1(\xi) + 2 \text{PD}(\Gamma)$ is torsion, where  $\Gamma$ denotes the homology class of both $\alpha$ and $\beta$, this does not depend on the choice of $Z$.
%The reason for the positive integer $d$ in the above equation is to make the relative grading well-defined, independent of a choice of relative homology class $Z$.  This is accomplished by taking $d$ to be the divisibility of $c_1(\xi) + 2 \text{PD}(\Gamma)$ in $H^2(Y;\mathbb{Z})$ modulo torsion, where  $\Gamma$ denotes the homology class of both $\alpha$ and $\beta$.  
It will be convenient in what follows to fix such a $\Gamma$ and define $ECH(Y,\lambda,\Gamma)$ to be the homology of the sub-complex generated by ECH generators with homology class $\Gamma$.
%; under the isomorphism \eqref{eqn:echswf}, this corresponds to fixing a spin-c structure for the Seiberg-Witten equations. 
%In particular, when  $c_1(\xi) + 2 \text{PD}(\Gamma)$ is torsion, we have $d = 0$ and so the relative ECH grading takes values in $\mathbb{Z}$.  
We can refine this relative grading to a (non-canonical) absolute grading $gr$ by declaring a fixed ECH generator to have grading $0$.  With this in mind, we can now state the {\em contact Weyl law}:

\begin{theorem} \cite{cghr}
\label{thm:contactweyl}
Let $(Y,\lambda,\Gamma)$ be a closed contact three-manifold such that $c_1(\xi) + 2 \text{PD}(\Gamma)$ is torsion.  Let $\lbrace \sigma_k \rbrace$ be any sequence of nonzero classes in $ECH(Y,\lambda,\Gamma)$ with gradings $gr$ tending to $+\infty$.  Then
\[ \lim_{k \to \infty} \frac{ c_\sigma(\lambda)^2}{gr(\sigma)} = \int_Y \lambda \wedge d \lambda.\]
\end{theorem}

\begin{remark}
\label{rmk:echswf}
A crucial point in applications is that whenever $c_1(\xi) + 2 \text{PD}(\Gamma)$ is torsion, an infinite sequence of nonzero classes $\lbrace \sigma_k \rbrace$ as in the statement of Theorem~\ref{thm:contactweyl} exists.  This is a highly non-trivial fact, which is a consequence of \eqref{eqn:echswf} and known properties of $HM$ \cite{km}.  In fact, one can always assume that $U \sigma_k = \sigma_{k-1}$; we call such a sequence a {\em $U$-sequence}.
\end{remark}

The contact Weyl law immediately 
%This immediately\
implies the Weyl law for four-manifolds, as we now explain.
\begin{proof}[Proof of Theorem~\ref{thm:cghr} assuming Theorem~\ref{thm:contactweyl}]
Normalize the grading so that $gr([\emptyset]) = 0$.  Because the $U$-map decreases the grading by $2$, the number $c_k$ is the spectral invariant associated to some class of grading $2k$.  Hence, Theorem~\ref{thm:cghr} follows from Theorem~\ref{thm:contactweyl}, after relating the right hand sides of each of the two Weyl laws via Stokes' Theorem.
%cany class w
\end{proof}

Thus it remains to explain the proof of Theorem~\ref{thm:contactweyl}.    We will give some sense of how this works below, but first we will take a detour, as we explain in the next section.
%want to explain some notable  applications.

\subsection{The Weinstein conjecture and the theorem on two orbits}   

Before continuing with explaining the proof of Theorem~\ref{thm:contactweyl}, let us sketch some notable applications.
% to Reeb dynamics.  

Let us first recall the {\em Weinstein conjecture}, a problem at the heart of developments in the field.  This states that a contact form on a closed manifold always has at least one periodic orbit.  (Recall that periodic orbits of the Reeb vector field are called Reeb orbits.)  In dimension $3$, this was proved by Taubes \cite{taubesweinstein}; in fact, the three-dimensional Weinstein conjecture follows from \eqref{eqn:echswf}, via Remark~\ref{rmk:echswf}.  Motivation for studying the Weinstein conjecture comes from Hamiltonian dynamics.  Namely, if the Hamiltonian $H$ is autonomous (i.e. does not depend on time), then by conservation of energy the dynamics preserve $H$; in many situations of interest these hypersurfaces $Y$ are {\em contact}, in the sense that $\omega|_Y = d \lambda$ for some contact form $\lambda$, and then the dynamics along $Y$ are 
%the restriction of $\omega
%and  dynamics are 
encoded 
by a Reeb vector field.
% see \cite{hutchingstaubes} for more detail.
% of a contact form 

In addition to the Weinstein conjecture itself, the problem of finding improved lower bounds on the number of periodic orbits is of wide interest.  For example, when a level set of $H:\mathbb{R}^{2n} \to \mathbb{R}$ is {\em star-shaped}, meaning it is transverse to the radial vector field, then the level set has a natural contact form and it is an old conjecture that the corresponding Reeb flow has at least $n$ Reeb orbits.  The first application of Theorem~\ref{thm:contactweyl} was the proof of the following refinement of Taubes' result in dimension 3, answering in particular the $n=2$ case of the above conjecture.
% for another proof of this conjecture in this dimension, see \cite{hginzburgetal}.

\begin{theorem}\cite{onetwo}
\label{thm:onetwo}
Any closed $3$-manifold with contact form $(Y,\lambda)$ has at least two simple Reeb orbits.
\end{theorem} 

Before giving the proof, let us note that this is the sharp result.  For example, as we explained in the proof of Proposition~\ref{prop:echell}, the natural contact form on the boundary of an irrational ellipsoid has exactly two embedded Reeb orbits.  We should also note that the general strategy of proof, of considerations of the growth rate of the asymptotics, comes up in various forms later in these notes as well.

%studying the asymptotics and applying growth rate considerat

\begin{proof}[Sketch of proof]
By the discussion above, we know that there is at least one orbit, so we just have to prove that there cannot be exactly one.  Assume there is exactly one orbit $\gamma$, with period $T$.   Consider a $U$-sequence $\lbrace \sigma_k \rbrace$ whose existence is guaranteed by Remark~\ref{rmk:echswf}.  %Let $y_k = c_{\sigma_k}(\lambda)$.  
Then each
\begin{equation}
\label{eqn:spec} 
 c_{\sigma_k}(\lambda) = m_k T, \quad \quad m_k \in \mathbb{Z}_{\ge 0}:
\end{equation}
in the non-degenerate case, this is immediate from the definition of the spectral invariant associated to $\sigma_k$ (since any orbit set is of the form $\lbrace (\gamma,m) \rbrace$), and the degenerate case follows along similar lines by approximation.   
In addition, one has the inequality
\begin{equation}
\label{eqn:decrease}
c_{U \sigma}(\lambda) < c_{\sigma}(\lambda).
\end{equation} 
%which implies in particular that $y_k > y_{k-1}$.  The relation $y_k > y_{k-1}$ 
This inequality then contradicts Theorem~\ref{thm:contactweyl}, since combining this relation with \eqref{eqn:spec} implies that each $c_{\sigma_k}(\lambda)$ is at least $T$ more than the previous one, hence since the $U$ map decreases the grading by $2$ one has
\[ \lim_{k \to \infty} \frac{c_{\sigma_k}(\lambda)^2}{gr(\sigma_k)} = + \infty.\]

Thus it remains to explain \eqref{eqn:decrease}.  Establishing this in the case where $\lambda$ is degenerate is the most technical part of the proof.  We will be brief, focusing on the geometric idea: the point is that the $U$ map counts $J$-holomorphic currents through a fixed point $(0,z)$ with $z$ not on any Reeb orbit.  In particular, since the curves are asymptotic to $\gamma$ but have to travel a fixed distance away from $\gamma$, this takes a certain amount of energy.  More precisely, such a current $C$ has $\int_C d \lambda > 0$, which is a straightforward exercise using the $\lambda$-compatibility of $J$ and the equation \eqref{eqn:jhol}.  In the degenerate case, one approximates by nondegenerate contact forms $\lambda_n$ and shows that the analogue $\int_{C_n} d \lambda_n$ still has a positive lower bound, independent of $n$; this is proved by a compactness argument;
%, using a compactness theorem for pseudoholomorphic currents due to Taubes; %\cite{taubescompactness}; 
the idea is that otherwise, the curves $C_n$ would have to converge to a zero energy curve through the point $(0,z)$, which can not occur. 
%   i.e. such a curve has a positive lower bound
% they have a positive lower bound $\int_{C_n} d \lambda > c$ independent of $n$
\end{proof}    

\subsection{Hidden boundaries}
\label{sec:hidden}

Here is a seemingly different sort of application which actually follows from a similar argument to that of the previous section in combination with previously known facts.

The following is an old question of Eliashberg-Hofer \cite{eliashhofer}: {\em to what degree does the interior of a symplectic manifold determine its boundary?}  Though a symplectomorphism might not extend to the boundary, there is nevertheless more ridigidity than one might expect.  
%The subtlety here is that a symplectomorphism might not extend to the boundary; however, as was observed in \cite{eliashhofer} there is nevertheless more ridigidity than one might expect.  
Eliashberg-Hofer specifically single out the case of the ball as a test case \cite{eliashhofer}, and construct some examples where the interior does not determine the boundary.  Using ideas similar to the proof of Theorem~\ref{thm:onetwo}, we can prove the following which essentially answers the Eliashberg-Hofer ball question in dimension $4$:
%about the ball:

\begin{theorem} \cite{cgm}
\label{thm:recog}
Let $(X,\omega)$ be a four-dimensional Liouville domain whose interior is symplectomorphic to the interior of a ball.  Then $(X,\omega)$ is symplectomorphic to a ball.
\end{theorem}

\begin{remark}
\label{rmk:echgeneral}
To prove the theorem, it is helpful to recall an extension of ECH capacities to arbitrary symplectic $4$-manifolds, due to Hutchings \cite{qech}: if $(X,\omega)$ is any symplectic $4$-manifold, Hutchings defines the ECH capacity $c_k(X,\omega)$ to be the supremum $c_k(X',\omega')$ over any Liouville domain embedded in the interior; it is an elementary exercise that this agrees with the usual ECH capacities for Liouville domains.
\end{remark}

\begin{proof}[Sketch of proof of Theorem~\ref{thm:recog}]
By Remark~\ref{rmk:echgeneral}, the ECH capacities are determined by the interior.  Hence, $(X,\omega)$ has the same ECH capacities as a ball.  Thus by Propositon~\ref{prop:echell}, two ECH capacities coincide and so, writing $\omega = d \lambda$ with $\lambda$ a contact form on the boundary, $c_{U \sigma}(\lambda) = c_{\sigma}(\lambda)$.  One should now compare to the argument used to prove \eqref{eqn:decrease}: with some more care one can show that \eqref{eqn:decrease} holds as long as one can place the point $z$ away from all Reeb orbits, so in particular since \eqref{eqn:decrease} does {\em not} hold in the present situation, we learn that every integral curve of the Reeb vector field for $\lambda$ is closed.  From this one can show that the Reeb flow in fact generates an effective circle action.  On the other hand, for topological reasons it must also be the case that $\partial X = S^3$, and the effective circle actions on $S^3$ have been classified.  It is not too hard to show from this that the Reeb flow on $\partial X$ in fact agrees with the Reeb flow on the boundary of the ball, and a computation using Moser's trick then implies that the contact form must be diffeomorphic to the contact form on the boundary of the standard ball.  However, the boundary of the ball is known to have a unique symplectic filling up to blow-down.
\end{proof}

\subsection{The closing lemma}
\label{sec:irie}

The following spectacular application of Theorem~\ref{thm:contactweyl} is due to Irie:

\begin{theorem} \cite{irie}
\label{thm:closing}
A $C^{\infty}$-generic contact form on a closed three-manifold $Y$ has a dense set of periodic orbits.  
\end{theorem}

To put this in context, Theorem~\ref{thm:closing} is an example of a {\em high regularity closing lemma}.  Such statements in low regularity  are known, see e.g. the discussion in \cite{ai} and the references therein.
% \cite{pugh}.
%due to celebrated work of Pugh and Pugh-Robinson \cite{pugh, pughrobinson}. 
 However, finding closing lemmas in higher regularity has long been known to be a very challenging matter for which new ideas are needed; this problem is featured as Smale's $10^{th}$ problem for the $21^{st}$ century.
 % \cite{smale}.  
 Let us now give the idea behind Irie's short argument:

\begin{proof}[Sketch of proof of Theorem~\ref{thm:closing}]

By a Baire category theory argument, it suffices to prove the following: fix an open set $U$, then by a $C^{\infty}$-small perturbation of $\lambda$ we can guarantee that there is a periodic orbit passing through $U$.  To prove this, for $0 \le t \le 1$, consider the perturbation $\lambda \to e^{tf} \lambda$, where $f$ is a small nonnegative nonzero function supported in $U$.  Fix a nonzero sequence of ECH classes $\lbrace \sigma_k \rbrace$ as in the statement of Theorem~\ref{thm:contactweyl}.   Then the first key observation of Irie is as follows: since $e^f \lambda$ has more volume than $\lambda$, it follows from the Weyl law Theorem~\ref{thm:contactweyl}  that for some $\sigma = \sigma_k$,
%\begin{equation}
%\label{eqn:irieequation}
$c_{\sigma}(e^{f} \lambda) > c_{\sigma}(\lambda).$
%\end{equation}
As Irie explains, this  now implies that $e^{tf} \lambda$ must have a periodic orbit passing through $U$ for some $0 \le t \le 1$; indeed, it is known that $c_{\sigma}(e^{tf}(\lambda))$ is continuous in $t$, and it is also known that the set of periods of the periodic orbits has measure $0$, hence $c_{\sigma}(e^{tf} \lambda)$ is taking values in a ($t$-dependent) measure $0$ set.  Thus, if $e^{tf} \lambda$ has no periodic orbits in $U$ for any $t$, this measure zero $t$-dependent set is constant in $t$ (since the periodic orbits away from $U$ are unchanged by the perturbation) and so  $c_{\sigma}(e^{tf}\lambda)$ must itself be constant, a contradiction.
%contradicting \eqref{eqn:irieequation}.
%Now the equation 
\end{proof}

\subsection{Low action curves of controlled topology}
\label{sec:low}

We now explain some applications of the contact Weyl law to understanding the structure of the pseudoholomorphic curves in four-dimensional symplectizations themselves.  These results will figure prominently in the more recent developments that we explain in \S\ref{sec:dynamics}.

To fix the setup, let $\lbrace \sigma_k \rbrace^{\infty}_{k=1}$ be a $U$-sequence, as guaranteed by Remark~\ref{rmk:echswf}.  Fx now a nondegenerate contact form $\lambda$.  By the existence of this $U$-sequence, for any $N$, we can find an ECH generator $\alpha(N)$, of action no more than $c_{\sigma_N}(\lambda)$, such that $U^N \alpha(N) \ne 0$; in particular, there exist ECH generators $\alpha(N), \ldots, \alpha(0)$ such that there are $J$-holomorphic curves $C(i)$, $1 \le i \le N$, counted by the ECH $U$-map, from $\alpha(i)$ to $\alpha(i-1)$.

While any individual $C(i)$ can be quite complicated --- for example, the ECH index condition gives no a priori bound on the genus of the curve --- we will be interested in the nature of the $C(i)$ ``on average".   Towards this end, we say that the {\em usual} ECH $U$-curve has some property $\mathcal{P}$ if the proportion of the $C(i)$ with that property tends to $1$ as $N \to \infty$.  Here is the first such property that we will be interested in.  Define the {\em action} of a curve $C$ by the formula
$\mathcal{A}(C) = \int_C d \lambda.$
This is a nonnegative quantity that is $0$ if and only if $C$ is $\mathbb{R}$-invariant.  It is loosely analogous to the height difference between two critical points connected by a flow line in Morse theory.

%If we think of $J$-holomorphic curves in symplectizations as loosely analogous to flow lines in mo

\begin{lemma}
The usual ECH $U$-curve has low action.  More precisely, fix any $\epsilon > 0$.  Then, the usual ECH $U$-curve has action $< \epsilon$
\end{lemma}

\begin{proof}
By Stokes' Theorem, we have
\[ \sum^N_{i=1} \mathcal{A}(C(i)) = \mathcal{A}(N) - \mathcal{A}(0) \le \mathcal{A}(N) \le c_{\sigma_N}(\lambda).\]
On the other hand, by Theorem~\ref{thm:contactweyl}, $c_{\sigma_N}(\lambda)$ is $O(N^{1/2})$ as $N \to \infty$.  The lemma now follows, since each $\mathcal{A}(C(i))  \ge 0$.
\end{proof}

To explain the next property of interest, recall that an ECH index $1$ curve is always a union of a single embedded component that is not $\mathbb{R}$-invariant and a number of components that are $\mathbb{R}$-invariant.  The same is true for ECH index $2$ curves \cite{echlecture} and we write such a curve $C = C_0 \sqcup C_1$ with $C_1$ denoting the non $\mathbb{R}$-invariant component.  We now want to prove the following coarse bound:

\begin{lemma}
\label{lem:topol}
The usual ECH $U$-curve has controlled topology when $c_1(\xi) = 0$.  More precisely, under this assumption, the usual ECH $U$-curve has $-\chi(C_1) \le 5$
\end{lemma}

Stronger results are possible in a similar spirit, but we want to keep the exposition as accessible as possible.  To prove this, we need to introduce a variant of the ECH index, originally introduced by Hutchings.  Define the $J_0$ index
\[ J_0(C) = - c_\tau(C) + Q_\tau(C) + CZ^J_\tau(C).\]
Here, the terms  $c_\tau(C)$ and $Q_\tau(C)$ are just as in the definition of $I$, and $CZ^J$ is a slight variant of its cousin in $I$, defined by the property that if $\alpha = \lbrace (\gamma_i,m_i)$, then $CZ_\tau^I -CZ_\tau^J(\alpha) =  \sum_i \lfloor m_i \theta_i \rfloor  + \lceil m_i \theta_i \rceil$, where the $\theta_i$ are the rotation numbers associated to the periodic orbits, see \cite{echlecture}.  The $J_0$ index depends only on the relative homology class of $C$.  On the other hand, Hutchings shows that
\begin{equation}
\label{eqn:J0top}
- \chi(C_1) \le J_0(C).
\end{equation}
Another important property proved in \cite{hutchingsj0} is that $J_0(C) \ge -1$. 

\begin{proof}[Sketch of proof of Lemma~\ref{lem:topol}]
Because the ECH index and the $J$ index are additive, we have 
\[ 2N = \sum^N_{i=1} I(C_i) = \sum J(C_i) + 2 c_\tau(C_i) + (CZ_\tau^I -CZ_\tau^J(\alpha(N)) - (CZ_\tau^I -CZ_\tau^J(\alpha(0)),\]
%where $\apha = 
%\sum_j \lfloor m_j \theta_i \rfloor  + \lceil m_j \theta_j \rceil$. \]
Choose now a trivialization of $\xi$, and use this to define the trivializations $\tau$.  Then it follows immediately that all $c_\tau$ terms vanish.  On the other hand, one can show that there exists a constant $C$ so that each term of the form  $\lfloor m_i \theta_i \rfloor$ is bounded in absolute value by the length of the corresponding orbit, up to a constant multiple of $C$; this is proved by bounding the rate of change of rotation, measured with respect to the global trivialization, as one follows the flow.  Since the length is bounded by the spectral invariant associated to $\alpha$, it then follows from Theorem~\ref{thm:contactweyl} and the above equation that $\sum J(C_i) - 2N$ has an $O(N^{1/2})$ bound.  Since each $J(C_i) \ge -1$, the Lemma follows. 
%For simplicity, let us assume first that 
\end{proof}

\subsection{The proof of the contact Weyl law}

We now try to give a sense of the proof of Theorem~\ref{thm:contactweyl} and explain some key ideas.   The proof builds on Taubes' proof of the isomorphism \eqref{eqn:echswf}.  Before getting into technicalities, let us give a sense for that proof from a high level point of view.  The Seiberg-Witten Floer cohomology is a kind of Morse homology, for a functional $\mathcal{F}$, the ``Chern-Simons-Dirac functional".  The idea of Taubes' isomorphism \eqref{eqn:echswf} is to perturb this functional in a direction determined by the contact form and rescale the perturbation to make the perturbation large.  Roughly speaking, letting $\mathcal{F}_r$ denote the perturbed functional, in the limit as we scale up the perturbation by a parameter $r$, Taubes shows that the critical points of $\mathcal{F}_r$ converge to ECH generators, the (index $1$) flow lines converge to pseudoholomorphic curves, and the corresponding gradings agree.  

\begin{remark}
In the spirit of Remark~\ref{rmk:minimal}, we note that in the minimal hypersurface context there is a picture that has interesting parallels with the above in the context of the Allen-Cahn equation; see e.g. \cite{cm}.
\end{remark}

Returning now to the proof, in what follows, it will be very useful to note that the critical points have the schematic form $(A(r),\psi(r))$ where $\psi$ is a section of rank $2$ Hermitian vector bundle over $Y$, and $A$ is a Hermitian connection on the determinant line bundle.

{\em The proof scheme:}  The basic idea of the proof of Theorem~\ref{thm:contactweyl} is now as follows.  Given a class $\sigma$ in ECH, we can choose a $1$-parameter family of critical points $p(r)$, for $\mathcal{F}_r$, such that $\mathcal{F}_r(p(r))$ is a continuous function of $r$; essentially, the $p(r)$ carry ``Seiberg-Witten spectral invariants" and this is used to produce this family.  There are then three key points: (i) We show that $v(r) :=  - \frac{2 \mathcal{F}_r(p(r))}{r}$ converges to $c_\sigma(\lambda)$ as $r \to \infty$; (ii) We show that $v(r)$ can be estimated for small values of $r$, essentially because the Seiberg-Witten equations can be solved for small value of $r$ when the grading of $\sigma$ is sufficiently large, and a Weyl law holds for these explicit solutions; (iii) we show that the rate of change $dv/dr$ can be understood up to error that washes out in the limit as the grading goes to $+\infty$.

Let us now try to give a little more detail about these three steps:

\begin{proof}[Remarks on the proof of Theorem~\ref{thm:contactweyl}]

We refer to the proof scheme above:

\vspace{2 mm}

(i): This follows from work of Taubes after some book keeping of filtered chain complexes.  The key geometric point, proved by the Taubes, is that the {\em energy} $E(A,\psi) :=  i \int \lambda \wedge F_A$
of a solution $(A,\psi) = (A(r),\psi(r))$ converges as $r \to \infty$ to the action of the corresponding ECH generator;  another estimate proved by Taubes shows that $v(r)$ and $E(r)$ have the same limit as $r \to \infty$.  Here $F_A$ denotes the curvature of the corresponding connection.  

\vspace{1 mm}

(ii): there are a few key points to make.  A critical point $(A,\psi)$ of $\mathcal{F}_r$ is called {\em reducible} if $\psi = 0$.  Now the first point is that because the unperturbed equations have only finitely many irreducible solutions \cite{km}, if the grading $j$ is sufficiently high then the critical points $p(r)$ of $\mathcal{F}_r$, determined as explained above by the class $\sigma$ in grading $j$ are {\em reducible} for sufficiently small $r$ depending on $j$, say $r < r_j$.
% with the key input for this the fact that the unperturbed equations have only finitely many irreducible solutions \cite{km}.  
Because in this case $\psi = 0$, one can actually solve for $(A(r),\psi(r))$ explicitly.  Of course we have $\psi(r) = 0$, and then the critical point equations reduce essentially to $F_A =  -2  i r d \lambda$.
%+ \ldots,$
%where the $\ldots$ represents some $r$-independent terms that are easy to control and so will not be discussed here.  
Thus one sees that at $r = r_j$ the energy is exactly the volume, up to a constant term depending on $r_j$.   To show that a Weyl law holds in this case, it remains to relate $j$ to $r_j$.  The grading on the Seiberg-Witten Floer homology is defined by spectral flow, and to do this one applies an estimate of Taubes \cite[Prop. 5.1]{taubesweinstein}.
%; one way to derive this estimate is to write the spectral flow as an index of a certain Fredholm operator, apply the Atiyah-Patodi-Singer index theorem, which relates this analytic index to a topological term, and estimate the error term by a heat kernel expansion \cite{cgs}.  %For brevity, we will not state the estimate here, but after applying this to untangle the relationship between $r, r_j$ and the topological term $cs(A)$, some elementary fiddling gives a Weyl law for reducibles.

(iii): one first notes that the functional $\mathcal{F}_r$ has the schematic form $\mathcal{F}_r = \frac{1}{2} ( cs(A)  - r E(r))$
where the energy is as above and $cs$ is another term, called the Chern-Simons term.  Setting aside some subtleties of differentiability since this is just a sketch,  because we are interested in critical points $p(r) = (A(r),\psi(r))$, we obtain that $\frac{d  \mathcal{F}_r}{dr} = -  \frac{1}{2}  E(A(r))$.   This gives a formula for $dv/dr$ in terms of $r$, $E(A(r))$ and $cs(A(r))$.  Some further estimates of Taubes are used to control  $E(A(r))$ and $cs(A(r))$, and some intricate comparison principles for ODEs beyond the scope of this survey are then used to get the necessary handle on $|dv/dr|$: what one finds is that there are two phases to $dv/dr$: it first changes by a substantial but controlled amount until some parameter $\tilde{r}_j$ and then further changes are very small; for the details see \cite{cghr}.
\end{proof}
%\begin{itemize}
%\item We show that $v(r) :=  - \frac{2 \mathcal{F}_r(p(r))}{r}$ converges to $c_\sigma(\lambda)$ as $r \to \infty$.
%\item We show that $v(r)$ can be estimated for small values of $r$, essentially because the Seiberg-Witten equations can be solved for small value of $r$ when the grading of $\sigma$ is sufficiently large, and a Weyl law holds for these explicit solutions.
%\item We show that the rate of change $dv/dr$ is not too large.
%\end{itemize}

\section{Weyl laws for surfaces and their applications}
\label{sec:ch2}

The next symplectic Weyl laws and their applications were for surfaces.  We now survey this story.

\subsection{Motivating questions} 

In the author's view, critical motivations for the next phase of development of the Weyl laws came from the following two longstanding questions:

\vspace{2 mm}

{\em The Simplicity Conjecture.}  The study of the algebraic structure of transformation groups has a long history; see for example the review in \cite{simp}.   While much has been understood about these questions, a stubborn case had been the case of area-preserving homeomorphisms of a surface.  In the 1970s, the work of Fathi \cite{fathi} explained the normal subgroup structure for the case of groups of volume-preserving transformations of a manifold of dimension at least three.  However, Fathi's proof failed in the two-dimensional case and he asked whether or not similar results nevertheless hold.  After some time, an intriguing conjecture emerged that the two-dimensional case should in fact contrast the other cases because in this case the volume form is symplectic.  In particular, the {\em Simplicity Conjecture} asserted that the group of area-preserving homeomorphisms of the standard (open) two-dimensional disc that are the identity near the boundary is not a simple group.  If true, this would contrast Fathi's results, which, in the case of the disc of dimension at least $3$ implies that the corresponding group is simple.

A closely related question, also due to Fathi, is whether or not the {\em Calabi homomorphism}, which is the obstruction to simplicity of the group of area-preserving {\em diffeomorphisms} of the two-disc that are the identity near the boundary, makes sense for homeomorphisms; an extension as a homomorphism would answer Fathi's question but it was known that Calabi can not just be extended by continuity. We will say more about the Calabi homomorphism in \S\ref{sec:calabi}; it a measurement of the ``average rotation" of the map.  Let us call this the {\em Calabi extension question.}

\vspace{3 mm}

{\em The Closing Lemma.}  As we explained in \S\ref{sec:irie}, where we also explained the context for the question, Irie had proven a spectacular closing lemma for three-dimensional Reeb flows.  The closing lemma for area-preserving diffeomorphisms of surfaces had long been desired
% --- Franks-Le Calvez \cite{fc} describe it as ``perhaps the
%most important” question about “the topological picture of the dynamics of $C^r$ generic area preserving
%diffeomorphisms of surfaces" ---
 and so it was natural to attempt this.  Asoaka-Irie \cite{ai} proved the closing lemma for ``Hamiltonian" diffeormphisms, i.e. those that arise as time-$1$ flows of Hamilton's equations of motion.
 % by a beautiful argument making use of the Weyl law in Theorem~\ref{thm:contactweyl}.  
 However, many area-preserving diffeomorphisms are not Hamiltonian and the sense was that to establish the full closing lemma in the (conservative) surface case, one would need further Weyl laws.

\vspace{3 mm}

Let us now explain the new Weyl laws that were introduced and explain how these questions were resolved.

\subsection{The link spectral invariant Weyl law.}
\label{sec:calabi}

We will now depart somewhat from the historical order of things for expository reasons.  To say a few words about the history first, though, the first work on Weyl laws for surfaces established a special case of a conjecture of Hutchings about spectral invariants defined by a cousin of embedded contact homology, called {\em periodic Floer homology (PFH) spectral invariants}; this Weyl law is also at the heart of the closing lemma, and was proved in full generality, so we will return to it in \S\ref{sec:closing}; the established special case was strong enough to settle the Simplicity Conjecture and this was the first proof \cite{simp}.  In this section, we instead concentrate on the {\em link spectral invariants} which came later.

Let us now get into specifics.  Denote by $G = \text{Homeo}_c(D^2,\mu_{std})$ the group of compactly supported area-preserving homeomorphisms of the standard two-disc.  This is what we want to show is not simple.  The link spectral invariants we will use are a sequence $f_d: G \to \mathbb{R}$
satisfying the following two key properties.
% which we state as theorems.  Here is the first:

%\begin{theorem} \cite{cghmss}
%\label{thm:quasi}
%The link spectral invariants are  homogeneous quasimorphisms of defect $2/d$, i.e.
%\begin{equation}
%\label{eqn:quasi}
%|f_d(ab) - f_d(a) - f_d(b) | \le \frac{2}{d}, \quad f_d(a^k) = k f_d(a) \quad \quad a, b \in G. 
%\end{equation} 
%\end{theorem}

\begin{theorem} \cite{cghmss}
\label{thm:quasi}
The link spectral invariants are  homogeneous quasimorphisms of defect $2/d$, i.e.
\begin{equation}
\label{eqn:quasi}
|f_d(ab) - f_d(a) - f_d(b) | \le \frac{2}{d}, \quad f_d(a^k) = k f_d(a) \quad \quad a, b \in G. 
\end{equation} 
When $g \in G$ is smooth they satisfy the Weyl law
\[ \lim_{d \to \infty} f_d(g) = Cal(g).\]
\end{theorem}

\begin{remark}
In fact, it had also been an open question of interest, called the ``Quasimorphism question" \cite{mcduffsalamon},  whether any non-trivial quasimorphisms on $G$ exist at all.  
%This was part of the motivation for defining and studying link spectral invariants.
\end{remark}

%The other key property is the Weyl law.  

Here, $Cal(g)$ denotes the aformentioned Calabi invariant of diffeomorphisms, defined as follows.  Let $g \in G$ be in addition smooth.  Then it is known in this setting that $g$ is Hamiltonian, as defined above, i.e. $g$ is the time-$1$ flow associated to some (possibly time varying) Hamiltonian $H(t)$; because $g$ is compactly supported, we can demand that $H = 0$ near the boundary.  We now define $Cal(g) := \int \int H dxdy dt.$
One can show that this does not depend on the choice of such $H$, after which it is not too hard to see that it is a surjective homomorphism.
% for an interpretation as the average rotation of $g$ see e.g. \cite{ghys} and the references therein. 
% We can now state the second key property.
%\begin{theorem} \cite{cghmss}
%\label{thm:linkweyl}
%Let $g \in G$ and assume in addition that $g$ is smooth.  Then
%\[ \lim_{d \to \infty} f_d(g) = Cal(g).\]
%\end{theorem}
Before continuing, let us give an example of a representative element of $G$ for which it was not clear how to extend Calabi.
%these kind of maps play a crucial role in the work we are surveying.

\begin{example}
\label{ex:infinite}
Consider the {\em twist map} $T_f$ defined by fixing $0$ and, in polar coordinates $(r,\theta) \in (0,1) \times [0,2\pi]$,   sending $(r,\theta) \to (r, \theta + f(r))$
where $f:(0,1) \to \mathbb{R}_{\ge 0}$ is some function that is zero near $1$.  Then $T_f \in G$, but it is not generally smooth; when $f$ is rapidly increasing with a singularity at $0$,  we call $T_f$ an {\em infinite twist}.
\end{example}

\subsubsection{Non-simplicity}
\label{sec:nonsimp}

We will say more about the definition of the link spectral invariants and their properties soon, but let us first explain how to resolve the Simplicity Conjecture and the  Calabi extension question from the above Theorem.  In other words, we explain the proof of the following:

\begin{theorem}
(i) The group $G = \text{Homeo}_c(D^2,\mu_{std})$ is not simple.  (ii)  The Calabi homomorphism $Cal: \text{Diffeo}_c(D^2,\mu_{std})$ extends as a group homomorphism to $G$. 
%\begin{itemize}
%\item The group $G = Homeo_c(D^2,dxdy)$ is not simple.
%\item The Calabi homomorphism $Cal: Diffeo_c(D^2,dxdy)$ extends as a group homomorphism to $G$
%\end{itemize}
\end{theorem}

Of course, item $(ii)$  implies item $(i)$, but as we will see in the proof, the extension in (ii) is quite non-canonical, so in the author's view it is pedagogically sounder to separate these statements and first explain $(i)$.

\begin{proof}

(i) Define $R' = \mathbb{R}^{\mathbb{N}}/\sim$, where $s \sim t$ if and only if $\lim s - t = 0$.  There is a natural map  
\[ F: G \to R', \quad \quad g \to (f_1(g), f_2(g), \ldots, ).\]
By Theorem~\ref{thm:quasi} this is a homomorphism.  Let $N$ denote its kernel.  This is a normal subgroup; it is also proper: it follows from Theorem~\ref{thm:quasi} that $N \ne G$, and  $N \ne 0$, since $G$ is not abelian.  Hence  $G$ is not simple.

(ii) There is a natural map 
\[ \Delta: \mathbb{R} \to  \mathbb{R}^{\mathbb{N}}/\sim, \quad \Delta(t) = (t, t, \ldots).\]
  Choose a section $s$ of this map by way of Zorn's lemma.  Then $s \circ F$ extends the Cal by Theorem~\ref{thm:quasi}, 

\end{proof}

\begin{remark}
As one sees from the proof, though the extensions used to establish item (ii) answer Fathi's question, they have a rather arbitrary form.  In fact, a certain amount of choice can not be avoided if one seeks an extension taking values in $\mathbb{R}$, though such an extension is not required to prove the Simplicity Conjecture.  Indeed, as explained to the author by Rosendal, there are models of ZF where the axiom of choice is false and any homomorphism $G \to \mathbb{R}$ would automatically be continuous; as we stated above, there can not be a continuous extension of Calabi.  As one sees from the proof, it is better to think of the extension of Calabi as taking values in $R'$.  We will say more about a very recent ``universal" extension due to Edtmair-Seyfaddini in \S\ref{sec:es}.

%Very recently, a ``universal" extension of Calabi taking values in 
\end{remark}

\subsubsection{Construction of the link spectral invariants}

Let us now give a sense of how the link spectral invariants are constructed in \cite{cghmss}.

\begin{remark}
Before getting into details,  let us make some remarks about continuous symplectic geometry.  A major challenge in applying symplectic techniques to Fathi's question is that this is a question about {\em homeomorphisms}, while the notion of symplectic map usually requires differentiability.  In fact, there is an extensive literature on applying symplectic techniques to continuous phenomena, see e.g. the references in \cite{simp}.  This is a beautiful story that is beyond the scope of our notes, but in what follows ideas from this field will play a key role.   One theme that we have already seen in the definition of ECH spectral invariants in the degenerate case is that quantities defined by Floer theory tend to have good continuity properties; this will come up again, though deeper techniques from continuous symplectic geometry are also required in what follows.
\end{remark}

Continuing now with the construction,  the general strategy, crucial in many constructions in continuous symplectic geometry, is that we first define them for diffeomorphisms and then extend them to homeomorphisms by continuity.
%So we first explain how to define them for diffeomorphisms.  
To start, we embed $D^2$ into $S^2$ as the southern hemisphere.  A basic principle in symplectic geometry is that it is very profitable to study {\em Lagrangian} submanifolds $L \subset M$; these are submanifolds on which the restriction of the symplectic form vanishes.  We are particularly interested in {\em non-displaceable} ones, i.e. manifolds that cannot be displaced from themselves by a Hamiltonian flow.  These have their own cohomology theory $HF(L)$ called {\em Lagrangian Floer cohomology} and in favorable situations there are associated spectral invariants after a choice of Hamiltonian on $M$, defined via a construction analogous to the construction in \S\ref{sec:spectral}.  On $S^2$, there is only one non-displaceable Lagrangian up to equivalence, namely an equator; we want a whole sequence of invariants so we would like to find more.  The idea in the construction of the link spectral invariants is then to find more Lagrangians as follows.  Call a collection $\mathcal{L}$ of $k$ disjoint curves in $S^2$ a {\em Lagrangian link}, and call this link {\em monotone} if it decomposes $S^2$ into segments of equal areas.   Fix, now, a choice of monotone Lagrangian link $\mathcal{L} = L_1 \sqcup \ldots L_k$.  We now consider the submanifold $T_k$ of the $k^{th}$-symmetric product $Sym^k(S^2)$ induced by the projection
\[  L_1 \times \ldots \times L_k \subset  (S^2)^k \to Sym^k(S^2).\]
This is a Lagrangian torus.  To be precise, the natural structure on $Sym^k(S^2) = \mathbb{C}P^k$ induced by the symplectic form on $S^2$ gives it the structure of a symplectic orbifold; however, by a construction due to Perutz, we can smooth this form out in a neighborhood of the diagonal to get a smooth symplectic form on $\mathbb{C}P^k$, for which $T_k$ is Lagrangian.   The monotonicity condition on the link ensures that $T_k$ is non-displaceable and monotone.   Let us now try to give a very impressionistic sense for how the $f_k$ are defined from the $T_k$:

{\em Step 1.}  We need to show that $HF(T_k) \ne 0$.   There is a much studied Lagrangian torus $C_k$ in $\mathbb{C}P^k$ called the Clifford torus; this is a fiber of the usual momentum map on $\mathbb{C}P^k$ and its Lagrangian Floer cohomology is the same as its singular cohomology.  We show that in fact $HF(T_k) = HF(C_k)$.  The Floer cohomology is defined by counting pseudoholomorphic discs; these discs are encoded in a ``disc potential", and to prove the claimed isomorphism of Floer cohomologies, it suffices to show these have the same disc potential.  The disc potential of the Clifford torus is well-known; the relevant discs in our case can be understood by the ``tautological correspondence", relating them to maps into $S^2$, and from this description the potential can be calculated and shown to agree with the Clifford one.

{\em Step 2.}  We now want to construct spectral invariants.  The construction is completely analogous to the construction in \S\ref{sec:spectral}.  As we said above, to specify the filtration on the Floer complex, we have to specify a Hamiltonian on $Sym^k(S^2)$.  Given an area-preserving diffeomorphism $g$, we write $g$ as the time $1$-flow of a Hamiltonian $H$;  this induces a Hamiltonian $Sym^k(H)$ on $Sym^k(S^2)$.   Thus to each non-zero class in $HF(T_k)$ we get an associated spectral invariant of $g$; by Step 1, $HF(T_k) = H_*(T_k)$ and we define $f_k(g)$ to be the spectral invariant associated to the fundamental class under this equivalence.  
%[Say why this doesn't depend on the choice of $H$]

{\em Step 3.}  We show that the $f_k$ are uniformly continuous --- we say more about this in the next section --- and so extend to homeomorphisms.

\begin{remark}
Our definition of link spectral invariants has a number of inspirations.  Most immediately, Polterovich-Sheulukhin defined closely related invariants in the case of $S^2$ \cite{ps} using a version of orbifold Floer cohomology that builds on work of Mak-Smith \cite{ms}, and used them to give another proof of the Simplicity Conjecture.  Going back earlier, our definition of link spectral invariants is reminiscent of similar constructions in Heegaard Floer cohomology, see e.g. \cite{osvath}.
\end{remark}

\subsubsection{Proofs of the key properties.}

Let us now try to give an idea of how the key properties of the link spectral invariants are proved:

\begin{proof}[Sketch of proofs]

{\em Weyl Law.}  The basic geometry is that as the number of components of the Lagrangian link tends to $+\infty$, the link sees the entire manifold.  To implement this, the key point is a {\em Lagrangian control} property that states that if $\mathcal{L} = \sqcup_i L_i$ is a monotone Lagrangian link and $H$ is a (possibly time-varying) Hamiltonian then the corresponding link spectral invariant $c_\mathcal{L}$ satisfies
\begin{equation}
\label{eqn:control}
  \frac{1}{k} \sum_i \int min_{L_i} H(t) dt \le c_\mathcal{L}(H) \le  \frac{1}{k} \sum_i \int max_{L_i} H(t).
  \end{equation}
The Weyl law then follows from this via the same kind of argument that shows that a Riemann sum converges to the corresponding integral.  The proof of the Lagrangian control property follows a standard line for Lagrangian spectral invariants associated to monotone Lagrangians and so we will skip it.

%More precisely, we prove a 

{\em Uniform Continuity.}   This uses some techniques from continuous symplectic geometry that are not at all related to Weyl laws and so beyond the scope of these notes; we will skip it for brevity.
% beyond the scope of these notes.  %The inputs from the Floer theory are a Triangle Inequality type property for link spectral invariants that is similar to the quasimorphism property, and a Hofer-Lipschitz property 

%Thus for brevity we will not say much.  The inputs from the Floer theory are a Triangle Inequality type property for link spectral invariants that is similar to the quasimorphism property, and a Hofer-Lipschitz property that states that each link spectral invariant is Lipschitz with respect to the distance in ``Hofer's metric" (we say more about this metric below), with Lipschitz constant $1$; see \cite[Prop. 3.4]{cghmss}.

{\em Quasimorphism.}  This also uses some properties that are not particularly related to Weyl laws and so also mostly beyond the scope of these notes.   The key point is an ``closed-open" map, relating the Floer cohomology of the Lagrangians in $\mathbb{C}P^n$ constructed via the link spectral invariants to the Floer theory of $\mathbb{C}P^n$; it was previously known that the latter can be used to define a quasi-morphism, and comparing the  link spectral invariants to this quasimorphism via the closed-open map gives Theorem~\ref{thm:quasi}.

\end{proof}

\subsection{Other surfaces and McDuff's perfectness question }

For other surfaces, there is a mostly parallel story.  To elaborate, let $S$ be such a surface, possibly with boundary, and let $\text{Homeo}_0(S,\omega)$ be the group of area-preserving homeomorphisms that are the identity near the boundary.  Fathi has constructed a {\em mass-flow} homomorphism $Homeo_0(S,\omega) \to H_1(S)/\Gamma,$
where $\Gamma$ is a certain discrete subgroup that we will not say much more about here, see \cite{fathi}.
% mass-flow is dual to the well-known flux homomorphism on the subgroup of smooth elements, and extends continuously.   The mass-flow also makes sense in higher dimensions, and Fathi has shown that in these dimensions its kernel is simple.   
The analogue of the above results for other surfaces is now the following:

\begin{theorem}\cite{cghmss}
The kernel of the mass-flow homomorphism is never simple.
\end{theorem}

The proof is somewhat similar to the proof in the case of the disc or sphere explained above.  One difference, however, is that the link spectral invariants no longer satisfy the quasimorphism property.  So, the proof in \cite{cghmss} goes via explicitly constructing a normal subgroup and showing that it is proper.  The construction of this normal subgroup follows the strategy in \cite{simp}, which was in fact the original proof of the Simplicity Conjecture.  % Here is an outline:

\begin{proof}[Outline of the proof] To briefly explain, the group of Hamiltonian diffeomorphisms of a symplectic manifold has a remarkable bi-invariant metric, called {\em Hofer's metric}; we will not review the definition here,  but the key fact for our purposes is that one can construct a normal subgroup, the group of ``finite Hofer energy homeomorphisms", which is essentially the largest normal subgroup subgroup such that one can still define Hofer's metric.  One then shows that membership in this group implies that the link spectral invariants remain bounded, so properness is proved by showing that a twist map as in Example~\ref{ex:infinite} can have unbounded asymptotics.
\end{proof}

For non-compact surfaces, if one drops the condition that the map is compactly supported then the group is never simple, because it contains the subgroup of compactly supported maps.  Nevertheless, interesting questions remain.   In particular, recall that a group is {\em perfect} if it is equal to its commutator subgroup.   The following is a longstanding question of McDuff's:

\begin{question} \cite{mcduff}
\label{que:mcduff}
%Let $\omega$ be any area-form on $\mathbb{R}^2$.  
%(\mathbb{R}^2,\omega)$
Is the group of area-preserving diffeomorphisms of $(\mathbb{R}^2,\omega)$ perfect?
\end{question}

There are two cases to the question, namely the infinite area and the finite area cases.  McDuff shows that in higher dimensions, the analogous groups are perfect.   In \cite{cghsjems}, we answered McDuff's question in the finite area case:

\begin{proof}[Sketch of proof]
Since the symplectic form has finite area, we can embed $(\mathbb{R}^2,\omega)$ in $S^2$ as the complement of the north pole $p_+$.  Consider an infinite twist centered at $p_+$ as in Example~\ref{ex:infinite}.  Its restriction to the complement of $p_+$ is a diffeomorphism.  If it was a commutator, then by the quasimorphism property Theorem~\ref{thm:quasi} its limit in the Weyl law Theorem~\ref{thm:quasi} would remain bounded; this is an easy exercise using the quasimorphism property.  Since one can construct a twist with unbounded asymptotics, the result follows.
\end{proof}

% the group of ``finite 
%5re is a subgroup 

\subsection{The Closing Lemma}
\label{sec:closing}

We will not say a huge amount about this for brevity, but we do want to make a few key points.

\subsubsection{Periodic Floer homology and PFH spectral invariants}  

The starting point for the proof is a cousin of ECH, called {\em periodic Floer homoloy} (PFH), also defined by Hutchings.  This is defined for area-preserving diffeomorphisms $\phi: (\Sigma,\omega) \to (\Sigma,\omega)$
of a closed surface.  To define $PFH(\phi)$, one first forms the mapping torus 
\begin{equation}
\label{eqn:mapping}
Y := \Sigma \times [0,1]_t \sim, \quad \quad (x,1) \sim (\phi(x), 0).
\end{equation}
This has a natural vector field $R := \partial_t$, a natural two-form $\omega_\phi$ induced by $\omega$, a natural two-plane field $V$ given by the vertical tangent bundle, and $\mathbb{R}_s \times Y$ carries a natural symplectic structure $ds \wedge dt + \omega_\phi$.  With this understood, one can now essentially copy the definition of ECH, recall \S\ref{sec:defnech}, to define PFH.

One should note, however, that $R$ is not in general the Reeb vector field for a contact form.  Instead, the natural structure on $Y$ is a pair $(dt,\omega_\phi)$, called a {\em stable Hamiltonian structure}.  In particular, in contrast to ECH, PFH does not have a natural interesting filtration.  An important insight of Hutchings is that to get interesting spectral invariants, one should instead work with a ``twisted" version $\widetilde{PFH}$ of PFH which mimics an analogous construction on ECH;  to define this, we choose a reference loop $\gamma$ in $Y$.   A generator of $\widetilde{PFH}(\phi,\gamma)$ is then a pair $(\alpha,Z)$, where $\alpha$ is a finite set of orbits of $R$, just as in the ECH case, such that $[\alpha] = [\gamma] \in H_1(Y)$, and $Z \in H_2(Y,\alpha,\gamma)$ is a relative homology class from $\alpha$ to $\gamma$.  The differential $\partial$ is defined just as in the ECH case, except that we now keep track of the homology class of the curve: the coefficient of $\partial (\alpha,Z)$ on $(\beta,Z')$ is a count of $I = 1$ currents $C$ from $\alpha$ to $\beta$, such that $[C] + Z' = Z$ in relative homology.  The same arguments as in the ECH case show that the homology is well-defined and arguments of Lee-Taubes \cite{leetaubes} show that it agrees with a version of Seiberg-Witten Floer theory, analogously to \eqref{eqn:echswf}.  The key point is that there is now a natural action filtration given by $\mathcal{A}(\alpha,Z) = \int_Z \omega_\phi.$
Given this, one can define spectral invariants just as in the ECH case and these are the PFH spectral invariants.  There is also a grading $I$ just as in the ECH case.

\subsubsection{The Weyl Law and the proof of the closing lemma}

We can now state the PFH Weyl law, which is the key ingredient in proving the closing lemma.  We think of this as a ``relative" Weyl law.  We will state a slightly simplified version, referring to \cite{cghpz} for the general statement.  Call the pair $(\phi,\gamma)$ {\em PFH monotone} if $c_1(V) + 2 PD([\gamma])$ is proportional to $[\omega_\phi]$ and define the degree of $\gamma$ to be its image in $H_1(S^1) = \mathbb{Z}$ under the projection of the mapping torus to $S^1$  

\begin{theorem} \cite{cghpz}
\label{thm:pfhweyl}
Fix a closed surface with area form $(S,\omega)$ and genus $G$.  Let $\varphi$ be an area-preserving diffeomorphism, fix a Hamiltonian $H$ supported in the complement of PFH monotone reference cycles $\Theta_m$ with degrees $d_m \to \infty$, and let   $\phi' = \phi \circ \psi^1_H$, where $\psi^1_H$ denotes the time-$1$ flow.  Then
\[ \lim_{k \to \infty} \frac{ c_{\tau_m}(\varphi') - c_{\sigma_m}(\varphi) }{d_m + 1 - G} - \frac{I(\tau_m) - I (\sigma_m) }{2 (d_m + 1 - G)^2} = \int_{\Sigma \times [0,1]} H \omega \wedge dt.\]
\end{theorem}

The proof follows a broadly similar line to the proof in \S\ref{sec:contactweyl}, though there are many subtleties that required new ideas; for a summary, see \cite{cghpz}.  Given this, to prove the Weyl law given this we need just one further ingredient.  %Call the pair $(\phi,\gamma)$ {\bf PFH monotone} if $c_1(V) + 2 PD([\gamma])$ is proportional to $[\omega_\phi]$ and define the degree of $\gamma$ to be its image in $H_1(S^1) = \mathbb{Z}$ under the projection of the mapping torus to $S^1$  

\begin{theorem}
\label{thm:nontriviality}
Let $(\phi,\gamma)$ be PFH monotone, with $\gamma$ of sufficiently large degree.  Then $\widetilde{PFH}(\phi,\gamma) \ne 0$ 
\end{theorem}

We should note that an assumption on $(\phi,\gamma)$ is certainly necessary for such a non-vanishing theorem.  For example, an irrational shift of a $2$-torus has no periodic points at all.  The proof of Theorem~\ref{thm:nontriviality} uses the Lee-Taubes isomorphism: via the analogue of \eqref{eqn:echswf}, it is equivalent to prove non-vanishing of the corresponding Seiberg-Witten Floer theory; it is known that in sufficiently large degree this Floer group is determined by the reducibles and can in fact be described in terms of classical algebraic topology.  Hence the theorem reduces to a classical topology fact which can be proved by a spectral sequence argument similar to the argument in \cite{km}.

We can now give the promised proof of the closing lemma, which essentially reproduces the Irie argument.  In other words, we prove the following:

\begin{theorem}
A $C^{\infty}$-generic area preserving diffeomorphism of a closed surface has a dense set of periodic points.
\end{theorem}

\begin{proof}[Sketch of Proof]
Given $\phi$, a straightforward computation shows that we can first approximate it by some $\phi'$ such that $\phi'$ admits PFH monotone pairs $(\phi',\gamma_d)$ with $d \to \infty$.  One now mimics the Irie argument: it suffices to show that given $U$ an open disc, we can perturb $\phi'$ so that it has a periodic point in $U$; we do this by composing $\phi'$ with a small Hamiltonian diffeomorphism supported in $U$ with positive Calabi invariant and arguing that if no periodic point appeared under composition with the time $t$-flow, each spectral invariant would be the same, in contradiction to the Weyl law Theorem~\ref{thm:pfhweyl}.
%$\omega_\phi'$
\end{proof}

\begin{remark}
Edtmair-Hutchings simultaneously and independently proved a related Weyl law in \cite{eh} and also used this to prove the closing lemma, with further refinements, assuming the existence of ``$U$-cyclic" classes in PFH; these were shown to exist in the monotone case (which suffices to prove the closing lemma) in \cite{cghpz}.

\end{remark}

%by counting $I = 1$ curves from $\aopha 

%maps 

%city 

%; we call such a via a sy

% area-preserving homeomorphisms of the disc

% but give a very conceptual proof.

%state a Weyl law 

% along with the new Weyl laws that 

\section{More recent developments}
\label{sec:ch3}

We now briefly survey some recent related developments.

\subsection{More elementary spectral invariants} 

The link, ECH, and PFH spectral invariants are all defined in terms of Floer homology.  However, it turns out that one can define related invariants using only pseudoholomorphic curves.  Specifically, Hutchings defines {\em elementary ECH spectral invariants} $c^{elem}_k$ in \cite{hutchingsjmd, hutchingspnas} and shows that these also satisfy a Weyl law like that in Theorem~\ref{thm:contactweyl}.  (It should be noted that the proof of this Weyl law uses Theorem~\ref{thm:contactweyl}). 
% In addition to having a more elementary definition,  the $c^{elem}_k$ have new applications; for example, Hutchings proves a     
Essentially, $c^{elem}_k$ is the minimum amount of action required so that for any $J$, there is a $J$-holomorphic curve in the symplectization passing through any collection of $k$ marked points.  Edtmair then similarly defines {\em elementary PFH spectral invariants} \cite{e} and shows that these can replace the PFH spectral invariants in the proof of the Simplicity Conjecture.   Both of these constructions are inspired by work of McDuff-Siegel, who introduced the idea of defining analogues of Floer-theroetic invariants this way in \cite{mcduffsiegel}.

%\subsection{ }
\subsection{The subleading asymptotics and applications}
\label{sec:subl}

In view of the above Weyl laws, it is natural to study the {\em subleading asymptotics}.    Many questions remain about what is true in this case but there have also been some interesting results and applications.

\subsubsection{Packing stability, complicated hidden boundaries, and the fractal Weyl law}

 Let us make this precise by first focusing on the case of ECH capacities.  Let $\text{Vol}$ denote the volume and define 
\[ e_k : = c_k - \sqrt{2 k \text{Vol}}.\]
It is then natural to ask about asymptotics of the $e_k$.  
We start by explaining an application to ``packing stability".  Recall from $\S\ref{sec:packing}$ that the packing numbers of the ball stabilize at $1$ for $k \ge 9$, a striking contrast from the Euclidean case.  In fact, this is a quite general phenomenon, called {\em packing stability}: celebrated work of Biran \cite{biran2} showed that the same holds for any closed symplectic $4$-manifold $(M,\omega)$ with $[\omega] \in H^2(M;\mathbb{Q})$; Buse-Hind \cite{bh} extended this result to arbitrary dimension, and Buse-Hind-Opshtein \cite{bho} extended Biran's result to all closed $4$-manifolds.  After this, it was natural to study the case with boundary and it had been an open question whether there were any symplectic manifolds at all for which packing stability failed.    The packing numbers are invariants of the interior, so the question naturally connects to the Eliashberg-Hofer question from \S\ref{sec:hidden} about the extent to which the boundary sees the interior.
%concerning the a closely     

Returning now to the $e_k$, the following is the starting point of \cite{cghjems}: {\em if the $e_k(X) \to -\infty$, then packing stability fails for $X$}.  This follows from the Monotonicity Property \eqref{eqn:mono}, because for any finite union of balls, the $e_k$ are readily shown to remain bounded in $k$ by Example~\ref{ex:ball} and the Disjoint Union Property \eqref{eqn:disjoint}.  The paper \cite{cghjems} then shows that the domains $X_p := \lbrace (z_1,z_2) |  |z_2|^2 < (1 + |z_1|^2)^{-p} \rbrace $ with $1 < p < 2$  have the property that the $e_k(X_p)$ are $O(k^{1/2p})$ as $k \to \infty$, approaching $-\infty$.  Thus packing stability fails for the $X_p$.  

In addition, these domains are shown to be symplectomorphic to bounded domains in $\mathbb{R}^4$. {\em What does their image look like?}  It is shown in  \cite{cghjems} that the condition on the growth rate of the $e_k$ forces any symplectomorphic image of the $X_p$ to have complicated boundary.  More precisely, the following ``fractal Weyl law" is proved:

\begin{theorem} \cite{cghjems}
Let $Z$ be a relatively compact and open subset of a symplectic four-manifold.  Then the inner Minkowski dimension $dim_M(\partial Z)$  of $\partial Z$ satisfies
\[ dim_M(\partial Z) \ge 2 + 4 \limsup_{k \to \infty} \left( \frac{ln( - e'_k(Z))}{ln(k)}\right),\]
where $e'_k(Z) = min(e_k(Z),0)$.
\end{theorem}

For example, it follows from this that the domains $X_p$, however they are compactified, have boundary with $dim_M(\partial Z) \ge 2 + 2/p$.  Given the above discussion, \cite{cghjems} asked what happens for domains with more regular boundary.  Recently, Edtmair proved the following remarkable result:

\begin{theorem} \cite{edtmair}
Any compact symplectic $4$-manifold with smooth boundary satisfies packing stability.  
\end{theorem} 

As a consequence of the above theorem, Edtmair is able to deduce $O(1)$ subleading asymptotics in a wide range of situations.  For example,  he is able to show that the $e_k$ are $O(1)$ for any compact smooth domain in $\mathbb{R}^4$;  for a much weaker bound on the ECH subleading asymptotics that holds in some other situations, see \cite{cgs}.  Edtmair is also able to show that PFH spectral invariants always have $O(1)$ subleading asymptotics for any smooth map.
%there is a ``fractal Weyl law" 

\subsubsection{New normal subgroups}
\label{sec:new}

There is a story for the algebraic structure of homeomorphism groups that is in some sense parallel to that of the previous section.   Recall the normal subgroup $\mathcal{N}$ from \S\ref{sec:nonsimp}.  It is natural to ask whether this is simple.  Considerations of the subleading asymptotics allow us to define some more subgroups.    %For notational simplicity, let us consider the case of $S^2$, referring to [ref, ref] for other surfaces.  
A key point is the following.   Recall the group $G = \text{Homeo}_c(D^2,\omega_{std})$ and consider the subgroup $G^{\infty}$ of smooth elements.  By the Weyl Law Theorem~\ref{thm:quasi}, any $g \in G^\infty \cap Ker(Cal)$ has the $f_k$ decaying to $0$.  Considerations of the decay rate allow us to define the following normal subgroup $S'$ of $G^\infty \cap Ker(Cal):$ we let $S'$ denote the subgroups of elements such that $k f_k$ remains bounded.  In fact, this is the entire group.  This follows from the much more general results of Edtmair mentioned above, but the first proof was the following simple argument:

\begin{proof}[Sketch of proof]
It is an easy exercise that homogeneous quasimorpshims are automatically conjugation invariant, so it follows from Theorem~\ref{thm:quasi} that $S'$ 
% the subgroup of the group of smooth diffeomorphisms consisting of smooth elements with $O(1)$ subleading asymptotics 
is a normal subgroup; one can show by considering rotations and computing with \eqref{eqn:control}  that it is non-empty.  However, Banyaga has shown that  the kernel of Calabi on the group of smooth diffeomorphism is simple, so $S' = G^\infty \cap Ker(Cal)$.
\end{proof}

To define a new normal subgroup, the idea is then as follows.   We first want to make sense of extending Calabi in a more canonical way than we did  in \S\ref{sec:nonsimp}. To do this, it is helpful to note that there is a normal subgroup $\text{Hameo}(D^2,\mu_{std}) \subset G$, the group of {\em hameomorphisms}, defined in \cite{ohmueller}; 
%for brevity, we will not give the definition here, but 
these are essentially those elements of $G$ that can be said to have a reasonable Hamiltonian associated to them.  The Calabi homomorphism then extends canonically to this group by integrating the Hamiltonian and we showed in \cite{cghmss} that this is well-defined.
% and therefore a homomorphism.  
 Now let $S$ be those elements of the kernel of this extension with $O(1)$ subleading asymptotics.
% as above. 

 The above computation for smooth elements shows that $S$ is non-trivial, and we can construct infinite twist maps as in Example~\ref{ex:infinite} which do not have $O(1)$  subleading asymptotics to show that it is proper; these twists are very analogous to the toric domains $X_p$ above and in fact inspired the construction of the $X_p$.  Thus we have constructed a new normal subgroup.   In fact, the question of simplicity of the kernel of Calabi on Hameo was raised by Oh-Mueller in \cite{ohmueller} and this construction resolves their question.  Of course it is natural to ask if this new normal subgroup is simple, but we do not know this; we will return to this in \S\ref{sec:open}.
 %There is a subgroup of {\em hameomorphisms} 

%$G = Homeo_(S^2,\omega_{std})$ be the group of area and orientation preserving homeomorphisms of $S^2$.   Let $S$ be the kernel of this map; we call this the group of {\em elements with $O(1)$ subleading asymptotics}. 

%In particular, we can consider the following variant of the map $F$ from XX.     Let $R'' = \mathbb{R}^{\mathbb{N}}/\sim,$ where $s \sim t$ if and only if $ \lim_k k (s - t)$ remains bounded.   %Now let
%\[ S = \lbrace g \in G |   
%As in XX, there is a natural map
%\[ G \to R'', \quad \quad g \to (f_1(g), f_2(g), \ldots, ).\]

%\begin{proof}

% This contains all smooth area-preserving diffeomorphisms: this follows, from example, by Edtmair's result explained in the previous section, but the first proof was actually a simple consequence of the quasimorphism property Theorem~\ref{thm:quasi}.  

%it follws 

%\[ 

%Let us explain the situation for $S^2$, referring to [ref] for some related results in higher genus.  
\subsection{Two or infinity and the Le Calvez - Yoccoz property}
\label{sec:dynamics}
We now explain some further applications to dynamics.  The proofs of these go well-beyond the relevant Weyl laws, and so to some extent will be beyond the scope of these notes, but we want to give some sense of how the Weyl laws come in.   

%We start with the following  question of Hofer-Wysocki-Zehnder, which had been longstanding.  Recall that $S^3$ has a {\em standard} contact structure, which is the kernel of the contact form on a ball.  They asked:
%\begin{question}
%{\em Does every Reeb flow on $S^3$ whose associated contact structure is standard have two or infinitely many simple Reeb orbits?}
%\end{question}  
%Hofer-Wysocki-Zehnder prove this for nondegenerate contact forms, such that the stable and unstable manifolds of all hyperbolic orbits intersect transversally, a generic condition.
% certain generic contact 
%We answered a generalization of Hofer-Wysocki-Zehnder's question in \cite{twoinf}, building on earlier work.% \cite{twoinfold}

\begin{theorem} \cite{twoinf, twoinfold}
Let $(Y,\lambda)$ be a closed three-manifold with contact form, such that the contact structure $\xi$ has $c_1(\xi) \in H^2(Y;\mathbb{Z})$ torsion. Then Reeb flow for $\lambda$ has two or infinitely many simple orbits.
\end{theorem}

This answers as a special case a longstanding conjecture of Hofer-Wysocki-Zehnder \cite{hwz}.  The proof is quite intricate so we will only say a few words, mainly trying to convey the overall strategy and explain where the Weyl law comes in.  

\begin{proof}[Remarks on the proof] The idea is to show that if there are only finitely many simple orbits, then the flow has an annular ``global surface of section" (GSS): this is an annulus with boundary on Reeb orbits, whose interior is embedded and transverse to the flow, and such that every flow line hits the surface both forwards and backwards in time, and the idea of finding global surface of sections via pseudoholomrphic curve techniques was pioneered by Hofer-Wysocki-Zehnder \cite{hwz}.   Given such a GSS, one can define a {\em first return map}, which maps a point to its next intersection under the flow with the GSS.   If one can arrange in addition that this map is smooth and area-preserving, then the theorem follows from a celebrated result of Franks \cite{franks}, which says that an area-preserving map of an annulus has infinitely many periodic points if it has any at all.        

To find this GSS, the general strategy is to approximate $\lambda$ with nondegenerate contact forms $\lambda_n$, show that each $\lambda_n$ has an annular GSS, and show that one can take a limit.  The GSS for each $\lambda_n$ is obtained by projecting an ECH $U$-map curve.   It is a remarkable fact that one can actually take a limit and keep a GSS in this situation; this is beyond the scope of these notes and it makes use of a kind of ECH miracle using the partition conditions; we refer the reader to \cite{twoinf} for a fuller discussion.  

The main place where the Weyl law comes in is in finding the GSS for each $\lambda_n$.  The proof of this builds on the discussion of low action curves of controlled topology in \S\ref{sec:low}.   There, some first considerations gave curves $C$ with some sort of bound.
% $ - \chi(C) \le 5$.  
However, in our proof, it turns out to be crucial to need cylinders.  It is not too hard to get curves with $- \chi(C) \le 2$ by an argument similar to the argument in \S\ref{sec:low}.  Getting cylinders, however, is much more delicate.  The basic idea is that there is an additional term in $J_0$ involving 
$\mathbb{R}$-invariant cylinders.  If there are at least two such cylinders for ``most" $U$-map curves, then the argument in \S\ref{sec:low} still works without much change to get cylinders.  Thus it suffices to argue that there are $U$-map curves with at least two $\mathbb{R}$-invariant cylindrical components.  The proof of this uses a new quantity called the ``score", which is tuned to the ``ECH partition conditions"; all of this is beyond the scope of these notes and we refer the reader to \cite{twoinf}.

Once one has the desired low-action cylinder counted by the ECH $U$-map for a $\lambda_n$, the basic idea is that it is not too hard to show that the moduli space of such cylinders is compact.  At the same time, an intersection theory argument, using the fact that $J$-holomorphic curves must intersect positively shows that no two distinct curves in the moduli space intersect.  Thus the curves sweep out the entire symplectization, and one can argue that since $J$ is admissible, any one of them projects to a GSS.   For the details, see \cite{twoinf}.
\end{proof}

In a different work, we show that the case of two orbits has substantial rigidity:

\begin{theorem} \cite{twoorbits}
Let $(Y,\lambda)$ be a closed three-manifold such that the associated Reeb flow has exactly two simple Reeb orbits.  Then $
\lambda$ is nondegenerate, $Y$ is a lens space, both orbits bound global surfaces of section, and the period of the two orbits satisfy $T_1T_2 = \frac{ \int _Y \lambda \wedge d \lambda}{ | \pi_1(Y) | }$.
\end{theorem}

For brevity, we will not say much about the proof; the Weyl law is used to show that the contact form is nondegenerate, which is the hardest part, by determining the rotation numbers of the orbits via their contribution to the grading; a kind of approximate ECH index in the degenerate case is useful for this. 

% Once one knows the contact form is nondegenerate, the strategy is to find a global surface of section like above.

The other dynamical result that we want to briefly recall goes beyond periodic orbits.  Recall that as explained above the union of the periodic orbits is generically dense.  In the non-generic situation, this certainly does not hold; in the irrational ellipsoid example in Proposition~\ref{prop:echell}, for example, there are exactly two periodic orbits.   On the other hand, periodic orbits are a special kind of {\em non-trivial invariant set}: this is a closed invariant set that is required to be proper and non-empty.  We showed the following:

\begin{theorem} \cite{cgp}
Let $(Y,\lambda)$ be a closed contact $3$-manifold such that $c_1(\xi) \in H^2(Y; \mathbb{Z})$ is torsion.  Then the union of the non-trivial invariant sets is dense for the Reeb flow.
%Let $(S,\omega)$ be a closed surface with area-form and let $\phi$ be any monotone area-preserving diffeomorphism.  Then the union of the  non-trivial invariant sets of $\phi$ is dense.
%\end{itemize}
\end{theorem}

We prove a similar result for surfaces.  These results are versions of a celebrated result of Le Calvez - Yoccoz for any homeomorphism (not necessarily area-preserving) of $S^2$ \cite{lcy}.  The proof, which uses the ``feral curve theory" of Fish-Hofer \cite{fh} and some further refinements due to Prasad  is again beyond the scope of these notes.   Where Weyl laws come in is in finding curves of controlled topology with actions going to $0$, along the lines of the arguments in \S\ref{sec:low}, generalized to the torsion or surface cases.  Given these curves, one proves a new compactness theorem, extracting nontrivial invariant sets in the limit.  For the details we refer the reader \cite{cgp}.

 %a dynamical system is called {\em minimal} if it does not have any closed proper invariant sets.  Any 

%\subsection{The y}

\subsection{The topological invariance of helicity}
\label{sec:es}

The {\em helicity} of an exact volume preserving flow $X$ on a three-manifold with volume-form $(Y,\mu)$ is the integral $\int_Y \alpha \wedge d \alpha,$ where $d \alpha = \iota_X \mu$; that such an $\alpha$ exists is the definition of the flow being exact.
%  %\pause
%\vspace{3 mm}
The helicity is 
%Motivation: 
the fundamental conserved quantity under the action of $\text{Diff}(Y,\mu)$. 
 %\pause 
 In the $70s$, Arnold asked if helicity extends to topological flows and is 
invariant under the action of the group of volume and orientation preserving homeomorphisms.
%$\text{Homeo}^+(Y,\mu)$.  
%As remarked by Tao [ref], this would be of interest, for example, in fluid dynamics.

Helicity is connected to the Calabi invariant as follows, as shown by Gambaudo-Ghys \cite{gg}:
Given $\psi \in \text{Diff}_c(D^2,\omega)$, let $Y$ denote the mapping torus as in \eqref{eqn:mapping}.
% = D^2 \times [0,1]_t / \sim  (x,1) \sim (\psi(x),0)$.
Then $\partial_t$ generates an exact volume preserving flow, and Gambaudo-Ghys show that its Calabi invariant is exactly its helicity.
%\vspace{3 mm}
%Then Helicity = Calabi. \pause 
%\vspace{3 mm}
Starting from this, and using the explained extension of Calabi, Edtmair-Seyfaddini recently answered Arnold's question for all flows without rest points in a remarkable work.   To do this, they construct a ``universal extension" of Calabi, taking values in $Homeo_c(D^2,\mu_{std})^{ab}$, where the superscript denotes the abelianization; this abelianization is non-trivial by arguments similar to the arguments we explained above to prove the Simplicity Conjecture, and one can think of it as a replacement for the range of the Calabi homomorphism on diffeomorphisms, which one normally regards as $\mathbb{R}$ but which is also $\text{Diffeo}_c(D^2,\mu_{std})^{ab}$ by Banyaga's theorem from above.  This is more canonical than the extensions previously discussed in \S\ref{sec:nonsimp}, which depend on a choice of link spectral invariants.  This facilitates keeping track of how inserting a mapping cylinder  like the one above into the flow changes the helicity; the fact that Calabi is homomorphism is crucial for this.  For more details, we refer the reader to \cite{es}.

%\pause 

%\vspace{5 mm}

%Tao: ``This would be of interest in fluid equations, as it would suggest that helicity remains invariant even after the development of
%singularities in the flow.``

%\end{frame}

%\begin{frame} \frametitle{Helicity and Calabi}

\section{Questions}
\label{sec:open}
We close with a few open questions related to possible directions for future developments of symplectic Weyl laws.

%Weyl laws.

\begin{question}
What more can be said about the subleading asymptotics?
\end{question}

As we reviewed in many cases these are now known to be $O(1)$.  One would  like to know if they extract a geometrically meaningful quantity in analogy with the volume extracted from the leading order term.  In the case of Weyl's original law, conjecturally one expects to extract the perimeter and this has been proven in the generic case \cite{weylivrii}.
 Hutchings has a heuristic
 %conjectured 
 \cite{hutchingsruelle} 
 %a heuristic 
recovering the 
%that, at least for generic Reeb flows giving the standard contact structure on $S^3$, the subleading asymptotics in the ECH case should recover the
  ``Ruelle" invariant, measuring the average rotation of the linearized flow, in some cases.  Similar questions can be asked about the other Weyl laws. 
In the case of toric domains, \cite{cgmm} shows that one can recover the affine perimeter in the convex case, without any genericity required.    
 % There are also interesting questions related to regularity.  For example, Edtmair has asked \cite{edtmair} if the assumption that $\partial X$ is smooth can be weakened to the assumption that $\partial X$ is $C^2$ in guaranteeing $O(1)$ subleading asymptotics; this is what one might expect in view of Hutchings' conjecture since the Ruelle invariant is $C^2$-continuous.  Edtmair constructs $C^{2 - \epsilon}$ examples with unbounded asymptotics.  Returning to surfaces, one can also ask if the subgroup $S$ in \S\ref{sec:new} is simple.  If so, it is equal to the commutator subgroup and this would give a characterization of commutators by subleading asymptotics.

\begin{question}
Is there anything like the above story in higher dimensions?
\end{question}

In fact for many of the kinds of applications we studied above, there are analogous questions one can ask in higher dimensions, but much less is known.  For example, the Weinstein conjecture remains open in higher dimensions;  it is open whether a star-shaped domain must have $n$ Reeb orbits; it is open whether or not packing stability can fail; and, the symplectic ball packing numbers of a ball are also not fully known.
 
%Philosophically, one might hope that there should be enough ``symplectic size measurements" in any dimension to recover the volume; however at present 

%Finally, we should note that nothing like the above story is currently known in higher dimensions;  it seems very natural to study what, if anything, can be said in higher dimensions that is in the spirit of these notes. 
 
\section*{Acknowledgments.}

My understanding of the mathematics discussed here has profited immensely from conversations with a number of mathematicians.   First and foremost, I want to thank all of my collaborators on the works surveyed here, who have taught me so much.  I also want to thank Oliver Edtmair, Richard Hind, Helmut Hofer, Michael Hutchings, Rohil Prasad and Rich Schwartz for very helpful feedback on a draft of these notes.  These notes are dedicated to my father, Martin Gardiner,  a devoted parent and scholar who passed away in $2023$.

\end{document}